\documentclass[10pt,onecolumn,final,reqno]{article}

\usepackage[margin=1in]{geometry}
\usepackage{graphicx}
\usepackage{subcaption}  
\usepackage{enumerate}
\usepackage{paralist}
\usepackage{xcolor} 
\usepackage[hypcap]{caption}
\usepackage[pdftex, colorlinks=true]{hyperref}		
\usepackage{tikz-cd}
\usepackage[sort&compress,numbers]{natbib}

\usepackage{titlesec}
\titleformat{\section}
  {\normalfont\large\bfseries}{\thesection}{1em}{}
\titlespacing*{\section}
  {0pt}{3.5ex plus 1ex minus .2ex}{3.3ex plus .2ex}

\titleformat{\subsection}
  {\normalfont\large\bfseries}{\thesubsection}{1em}{}
\titlespacing*{\subsection}
  {0pt}{*5}{*1}

\usepackage{amsmath}
\usepackage{amsthm}
\usepackage{amsfonts}
\usepackage{amssymb}
\usepackage{mathrsfs}  
\usepackage{mathtools}
\usepackage{bm} 
\usepackage{bbm} 

\newcommand{\mc}{\mathcal}
\newcommand{\msf}{\mathsf}
\newcommand{\mbf}{\bm} 			

\newcommand{\N}{ \ensuremath{\mathbb{N}}}  

\newcommand{\R}{ \ensuremath{\mathbb{R}}}  
\newcommand{\C}{ \ensuremath{\mathbb{C}}}  

\newcommand{\eps}{\epsilon}

\providecommand\given{} 
\newcommand\SetSymbol[1][]{
   \nonscript\,#1\vert \allowbreak \nonscript\,\mathopen{}}

\DeclarePairedDelimiterX\set[1]{\lbrace}{\rbrace}{ \renewcommand\given{\SetSymbol[\delimsize]} #1 }  

\DeclarePairedDelimiterX\norm[1]{\lVert}{\rVert}{#1}  			
\DeclarePairedDelimiterX\inner[2]{\langle}{\rangle}{#1 \,,\, #2}  	
\DeclareMathOperator{\diag}{diag}								

\DeclarePairedDelimiterX\abs[1]{\lvert}{\rvert}{#1} 

\theoremstyle{plain}
  \newtheorem{theorem}{Theorem}[section]
  \newtheorem{lemma}[theorem]{Lemma}

  \newtheorem{corollary}[theorem]{Corollary}
  \newtheorem{condition}[theorem]{Condition}
\theoremstyle{remark}
  \newtheorem*{remark}{Remark}

\theoremstyle{definition}
  \newtheorem{example}[theorem]{Example}
  \newtheorem{definition}[theorem]{Definition}

\usepackage[T1]{fontenc}

\begin{document}
\title{Koopman Spectrum for Cascaded Systems}
\author{Ryan Mohr\thanks{Ryan Mohr (corresponding author). AIMdyn, Inc., Santa Barbara, CA 93101. Email: \texttt{mohrr@aimdyn.com}}~ and Igor Mezi\'{c}\thanks{Igor Mezi\'{c}. AIMdyn, Inc., Santa Barbara, CA 93101. Department of Mechanical Engineering, University of California, Santa Barbara, CA 93106.}}
\date{}

\maketitle

\begin{abstract}
This paper considers the evolution of Koopman principal eigenfunctions of cascaded dynamical systems. If each component subsystem is asymptotically stable, the matrix norms of the linear parts of the component subsystems are strictly increasing, and the component subsystems have disjoint spectrums, there exist perturbation functions for the initial conditions of each component subsystem such that the orbits of the cascaded system and the decoupled component subsystems have zero asymptotic relative error. This implies that the evolutions are asymptotically equivalent; cascaded compositions of stable systems are stable. These results hold for both cascaded systems with linear component subsystem dynamics and linear coupling terms and nonlinear cascades topologically conjugate to the linear case. We further show that the Koopman principal eigenvalues of each component subsystem are also Koopman eigenvalues of the cascaded system. The corresponding Koopman eigenfunctions of the cascaded system are formed by extending the domain of definition of the component systems' principal eigenfunctions and then composing them with the perturbation function.
\end{abstract}

\tableofcontents

\section{Introduction and terminology}
Engineered systems are becoming larger and more complex. Additionally, new tools allow us to measure the detailed behavior of increasingly complicated systems, such as biological-chemical networks. This growing complexity presents challenges to traditional tools of modeling and analysis, significantly slowing insight into the behavior of these systems. Predicting the behavior of these systems is hard or insoluble with these tools. For example, any chaotic behavior enforces limits on the possible prediction of the behavior of individual trajectories. This often requires the analysis of statistical properties of ensembles of trajectories \cite{Lasota:1994vt}. Unfortunately, the numerical computation of trajectories from initial conditions can be intractable due to the memory requirements needed to represent the system. New tools for the modeling, analysis, and prediction of complex systems must be developed.

The Koopman (composition) operator associated with a dynamical system is one such tool that has proven useful in analyzing complex systems (see \cite{Mezic:2000tm,Mezic:2004is,Mezic:2005ug} and \cite{Budisic:2012cf} and its references). The Koopman operator is an infinite dimensional linear operator that can fully capture the nonlinear behavior of a dynamical system. Much of the analysis of the system is done by spectrally decomposing observables on the system of interest into eigenvalues, eigenfunctions, and modes (called Koopman modes) of the Koopman operator \cite{Mezic:2004is,Budisic:2012cf}. Mathematically, the Koopman modes are the vector-valued coefficients in the expansion of a vector-valued observable into the (scalar-valued) eigenfunctions of the operator. Physically, the Koopman modes correspond to structures that are observable in an experiment. This has been used with great success in areas such as fluids dynamics (see \cite{Mezic:2013ei} and the references therein). A strength in using this operator is that its spectral decomposition can be approximately computed using data (numerical or experimental), without access to equations of motion for the system \cite{Budisic:2012cf,Williams:2015kh}. Recently, Koopman operator techniques have been extended to control-related problems such as system identification, observer design, and stability analysis \cite{Mezic:2015bl,Mauroy:2014ul,Sootla:bz,Mauroy:fh,Mauroy:cf}

However, to prove concrete theorems on the correspondence between the behavior of the Koopman operator and the behavior of the underlying system, some structure to that system must be assumed. Many engineered and natural systems exhibit a modular type structure where simple functional units are composed into more complex systems \cite{Callier:1976ky,Michel:1983dc,Pichai:1983be,Mezic:2004ur,Lan:2011hk,ShenOrr:2002jo,Mesbahi:2015es}. These systems can often be organized into a forward production unit with slower feedback loops modulating the forward unit's behavior. The forward unit displays a cascade-type structure. Modules farther downstream (modulo the slow feedback loops) do not influence the behavior of modules upstream to themselves. Given a model of a system, decomposition into modular components can be identified using various proposed techniques \cite{Pichai:1983be,Mezic:2004ur,Banaszuk:2011fy,Mesbahi:2015es}.

In all of these examples, there is a network structure where simple components are wired together to produce more complex systems. The procedure of wiring simpler dynamical systems together into more complicated systems can be described in language of the operad of wiring diagrams, where component systems are treated as black boxes, with manifold-typed input/output ports \cite{Vagner:2015tm,DeVille:2015hv}. A wiring diagram then specifies how boxes are wired together to form new boxes. Finally, an open dynamical system can be assigned to each box.

In this paper, we are interested in the behavior of the principal Koopman eigenfunctions corresponding to the forward production unit. In general, we consider a dynamical system constructed via an operadic-type wiring of component dynamical systems. Specifically, we look at dynamical systems with a cascade-type structure --- subsystem $i$ is unaffected by the dynamics of any subsystem $j$ if $j > i$ (see Fig. \ref{fig:cascade-system}). The resulting system has a lower block triangular wiring structure. The goal of this paper is to understand how the Koopman eigenfunctions of the component subsystems in a cascaded system behave under the action of the Koopman operator associated with the full system. In other words, how similar are the evolutions of the principal eigenfunctions when evolved with the cascaded system's dynamics and when evolved by just the component system dynamics? Additionally, what is the relation of the Koopman principal eigenvalues of the component subsystems to the Koopman spectrum of the cascaded system?

We show that for cascades with stable component subsystems whose eigenvalues satisfy a non-resonance condition between component subsystems and whose upstream systems are faster than the downstream systems that the Koopman principal eigenvalues of the component subsystems are also Koopman eigenvalues of the cascaded system. Furthermore, the associated Koopman eigenfunctions of the cascaded system are given by composing the principal eigenfunctions of the component subsystem with a perturbation function (that we give explicitly) that maps initial conditions to initial conditions. These results follow from a stronger one that we prove; the dynamics of the full system and the dynamics of the decoupled system converge exponentially faster to each other than the decoupled component system converges to its fixed point. In this case we say that the system has zero asymptotic relative error. These results allow one to analyze in some detail the behavior of a large cascaded system, without actually having to simulate it.

The rest of this section is devoted to precisely defining the cascaded systems and their associated nominal system (sec.\ \ref{subsec:cascaded-systems}); solution operators for the systems (sec.\ \ref{subsec:solution-operator}); the concepts of asymptotic proportionality, asymptotic equivalence, and zero asymptotic relative error (sec.\ \ref{subsec:asymptotic-equivalence}); and defining the Koopman operators and principal eigenfunctions of the component subsystems and extending their definitions to the full cascaded system using a tensor product construction (sec.\ \ref{subsec:koopman-operator}). Section \ref{sec:main-results} states the main theorems of the paper and their corollaries along with a number of remarks. Results of a numerical simulation supporting the main theorems is supplied as well. Section \ref{sec:proofs} collects the proofs of the main theorems in its subsections.

\subsection{Cascaded systems.}\label{subsec:cascaded-systems}
Given $x_{i}(t) \in \C^{d_{i}}$ for $i=1,\dots, n$, an $n$-level cascaded system is defined to have generators $\mbf N_i : \C^{d_1}\times \cdots \times \C^{d_i} \to \C^{d_i}$ of the form
	\begin{equation}
	\begin{aligned}\label{eq:nonlinear-cascade}
	x_1(t+1) &= L_1 x_1(t) + N_1(x_1(t))& \\
	x_i(t+1) &= L_i x_i(t) + \sum_{j=1}^{i-1} C_{i,j}x_j(t) + N_i(x_1(t), \dots, x_i(t)),& & (i=2,\dots, n),&
	\end{aligned}
	\end{equation}
where $L_{i} : \C^{d_i} \to \C^{d_i}$ and $C_{i,j} : \C^{d_j} \to \C^{d_i}$ are linear operators and $N_i : \C^{d_1} \times \cdots \times \C^{d_i} \to \C^{d_i}$ are nonlinear operators (see Fig.\ \ref{fig:cascade-system}). Each generator $\mbf N_i$ is called the component system of the nonlinear cascade. For any fixed $j \in \set{1,\dots, n}$, if $i < j$, then $\mbf N_i$ is called an upstream system to $\mbf N_j$ and if $k > j$, we call $\mbf N_k$ a downstream system to $\mbf N_j$. Each component system $\mbf N_i$ depends only upon itself and its upstream systems.

In proving our results, it will be beneficial to analyze the linearized system,
	\begin{equation}
	\begin{aligned}\label{eq:linear-cascade}
	x_1(t+1) &= L_1 x_1(t)\\
	x_i(t+1) &= L_i x_i(t) + \sum_{j=1}^{i-1} C_{i,j}x_j(t),& 	&(i=2,\dots, n)
	\end{aligned}
	\end{equation}
and the associated nominal linear system which is obtained by setting the linear coupling terms $C_{i,j}$ to 0:
	\begin{align}\label{eq:nominal-linear-system}
	x_i(t+1) & = L_i x_i(t), \qquad (i=1,\dots, n).
	\end{align}

\begin{figure}[h]
\begin{center}
\includegraphics[width=0.9\textwidth]{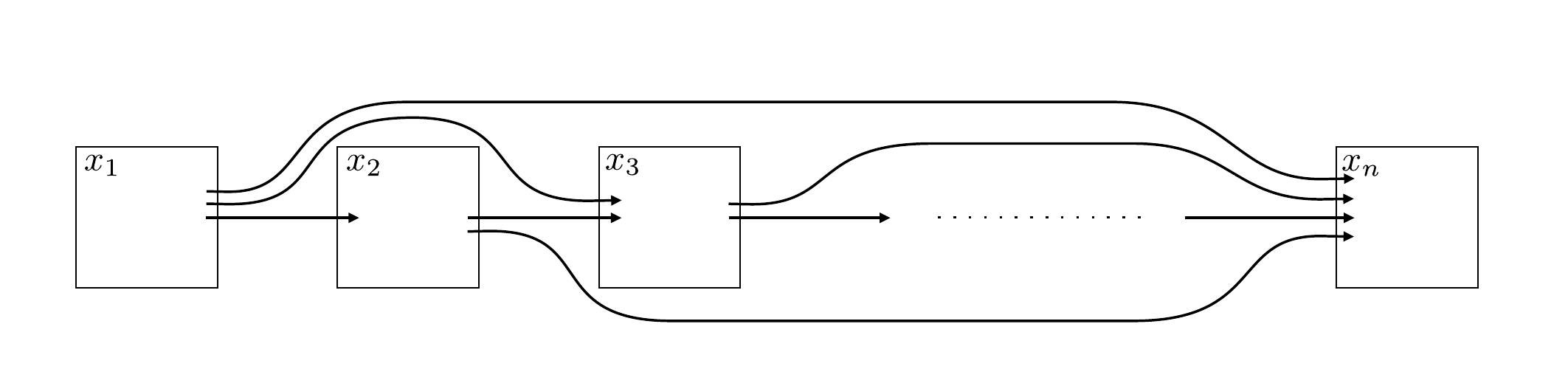}
\caption{$n$-level cascade system}
\label{fig:cascade-system}
\end{center}
\end{figure}

\subsection{Solution operators.}\label{subsec:solution-operator}
For $i=1,\dots, n$, let $\Pi_{i} : \C^{d_1} \times\cdots\times \C^{d_n} \to \C^{d_i}$ denote the canonical projections
	\begin{align}\label{eq:canonical-projection}
	\Pi_{i}(x_1,\dots,x_n) = x_i.
	\end{align}
The orbit $\msf{NonLin} : \C^{d_1} \times \cdots \times \C^{d_n} \to \C^{d_1} \times \cdots \times \C^{d_n}$ for the nonlinear cascaded system \eqref{eq:nonlinear-cascade} is denoted as\footnote{In \eqref{eq:vector-valued-function-notation} and later, we are using the notation that if we have a collection of maps $f_i : X \to X_i$, $(i=1,\dots,n)$, the vector-valued map $\mbf f : X \to X_1\times\cdots\times X_n$ defined by $\mbf f(x) := (f_1(x),\dots,f_n(x))$ is written as $f(x) \equiv (f_1,\dots, f_n)(x)$.}
	\begin{align}
	\msf{NonLin}^{\circ t}(x_1, \dots, x_n) 
	&= (\mc N_1^{\circ t}(x_1), \mc N_2^{\circ t}(x_1, x_2), \dots, \mc N_n^{\circ t}(x_1,\dots,x_n)) \label{eq:nonlinear-cascade-solution}\\
	&= (\mc N_1^{\circ t}\circ \Pi_1, \mc N_2^{\circ t}\circ (\Pi_1, \Pi_2), \dots, \mc N_n^{\circ t} \circ (\Pi_1,\dots, \Pi_n) )(x_1,\dots,x_n) \label{eq:vector-valued-function-notation}.
	\end{align}
where for each $i=1,\dots, n$, $\set{\mc N_{i}^{\circ t} : \C^{d_1} \times \cdots \times \C^{d_i} \to \C^{d_i}}_{t\in\N_0}$ is the family of solution operators for the $i$-th system. That is, given $(x_1,\dots, x_n)$, $\mc N_i^{\circ t}(x_1,\dots, x_i) \equiv x_i(t)$ is the value in the $i$-th system at time $t$ of the orbit passing through initial condition $(x_1,\dots,x_i)$. For $i=1,\dots, n$, define $\mc N_i^{\circ 0} = I_{\C_{d_i}}$, where $I_{\C_{d_i}}$ is the identity operator on $\C^{d_i}$. Then, for all $i \geq 1$ and $t \in \N$, these operators satisfy
	\begin{equation}
	\begin{aligned}
	\mc N_{i}^{\circ t}(x_1,\dots,x_i) &= \Pi_i \circ \msf{NonLin}^{\circ t}(x_1,\dots,x_n) \\
	&= \Big(L_i \circ\Pi_i + \big(\sum_{j=1}^{i-1} C_{i,j}\circ\Pi_j\big) + N_i \circ (\Pi_1,\dots, \Pi_i) \Big) (\msf{NonLin}^{\circ t}(x_1,\dots,x_n)).
	\end{aligned}
	\end{equation}
For $t=1$ we denote $\msf{NonLin}^{\circ 1} \equiv \msf{NonLin}$.
For $i\leq j$, define the $(i,j)$-slice of $\msf{NonLin}$ as
	\begin{equation}\label{eq:nonlinear-section}
	\begin{aligned}
	\msf{NonLin}_{i:j}^{\circ t}(x_1, \dots, x_n) 
	&= (\Pi_i, \Pi_{i+1},\dots, \Pi_j)(\msf{NonLin}^{\circ t}(x_1, \dots, x_n)) \\
	&= (\Pi_i \circ \msf{NonLin}^{\circ t}, \Pi_{i+1} \circ \msf{NonLin}^{\circ t}, \dots, \Pi_j\circ \msf{NonLin}^{\circ t})(x_1, \dots, x_n).
	\end{aligned}
	\end{equation}
If $i=j$, we will denote $\msf{NonLin}_{i:j}$ as $\msf{NonLin}_{i}$.

The orbit for the linear cascade \eqref{eq:linear-cascade} is defined as
	\begin{equation}\label{eq:linear-cascade-solution}
	\begin{aligned}
	\msf{Lin}^{\circ t}(x_1, \dots, x_n) 
	&= (\mc L_1^{\circ t}(x_1), \mc L_2^{\circ t}(x_1, x_2), \dots, \mc L_n^{\circ t}(x_1,\dots,x_n))
	\end{aligned}
	\end{equation}
where $\set{ \mc L_i^{\circ t} : \C^{d_1} \times \cdots \times \C^{d_i} \to \C^{d_i}}_{t\in\N_0}$ is the family of solution operators for system $i$ which satisfy
	\begin{equation}
	\begin{aligned}
	\mc L_{i}^{\circ t}(x_1,\dots,x_i) &= \Pi_i \circ \msf{Lin}^{\circ t}(x_1,\dots,x_n) \\
	&= \Big(L_i\circ\Pi_i + \sum_{j=1}^{i-1} C_{i,j}\circ\Pi_j \Big) (\msf{Lin}^{\circ t}(x_1,\dots,x_n)).
	\end{aligned}
	\end{equation}
The orbit for the full nominal system is
	\begin{equation}\label{eq:nominal-system-solution}
	\begin{aligned}
	\msf{Nom}^{\circ t}(x_1,\dots, x_n) 
	&= (L_1^{t}(x_1), L_2^{t}(x_2), \dots, L_n^{t}(x_n))
	\end{aligned}
	\end{equation}
where $L_i$ is the same as in \eqref{eq:linear-cascade}.
We will define the $(i,j)$-slices of $\msf{Lin}$ and $\msf{Nom}$ analogous to \eqref{eq:nonlinear-section}.

\subsection{Asymptotic equivalence.}\label{subsec:asymptotic-equivalence}
The main results of this paper rely on the concept of asymptotic proportionality and asymptotic equivalence between the component systems of the linear cascade and the corresponding nominal component system. First we define the norm for a cascaded system

\begin{definition}
Let $\norm{\cdot}_{\C^{d_i}}$ be a norm for $\C^{d_i}$. Define a norm on $\C^{d_1} \times\cdots\times \C^{d_n}$ by
	\begin{align}\label{eq:composite-norm}
	\norm{(x_1,\dots,x_n)}_{\times} = \sum_{i=1}^{n} \norm{x_i}_{\C^{d_i}}.
	\end{align}
\end{definition}

For linear maps and operators from one vector space to another, we will denote their induced norm as $\norm{\cdot}$. The norms used in the definition of the induced norm will be clear from the context.
 
\begin{definition}\label{def:asymptotic-equivalence}
We say that $\msf{Lin}$ is asymptotically equivalent to $\msf{Nom}$ if there exists a perturbation function $\msf{pert} : \C^{d_1}\times \cdots \times \C^{d_n} \to \C^{d_1}\times \cdots \times \C^{d_n}$ such that for all $i=1,\dots, n$ and all $x = (x_1,\dots,x_n) \in \C^{d_1}\times \cdots \times \C^{d_n}$,
	\begin{equation}\label{eq:asymptotic-equivalence}
	\lim_{t\to\infty} \norm{ \Pi_i \circ \msf{Lin}^{\circ t}(x) - \Pi_{i} \circ \msf{Nom}^{\circ t}(\msf{pert}(x)) }_{\C^{d_i}}  = 0.
	\end{equation}
We say that system $i$ has 0 asymptotic relative error if
	\begin{equation}\label{eq:asymptotic-relative-error}
	\lim_{t\to\infty} \frac{\norm{ \Pi_i \circ \msf{Lin}^{\circ t}(x) - \Pi_{i} \circ \msf{Nom}^{\circ t}(\msf{pert}(x)) }_{\C^{d_i}} }{\norm{\msf{Nom}_{i}}^{t}} = 0.
	\end{equation}
\end{definition}

A related concept is the asymptotic proportionality of two solution operators. It is the easier of the two to satisfy:

\begin{definition}\label{def:asymptotic-proportional}
We say that $\msf{Lin}$ is asymptotically proportional to $\msf{Nom}$ if there exist a perturbation function $\msf{pert} : \C^{d_1}\times \cdots \times \C^{d_n} \to \C^{d_1}\times \cdots \times \C^{d_n}$ and functions $\zeta_i : \C^{d_1}\times \cdots \times \C^{d_{i-1}} \to \R^+$ such that for all $x = (x_1,\dots,x_n) \in \C^{d_1}\times \cdots \times \C^{d_n}$ and all $i \in \set{2,\dots,n}$,
	\begin{equation}\label{eq:asymptotic-proportionality}
	\lim_{t\to\infty} \frac{\norm{ \Pi_i \circ \msf{Lin}^{\circ t}(x) - \Pi_{i} \circ \msf{Nom}^{\circ t}(\msf{pert}(x)) }_{\C^{d_i}} }{\norm{\msf{Nom}_{i}}^{t}}
	\leq \zeta_i(x_1,\dots,x_{i-1}).
	\end{equation}
\end{definition}

For the linear cascade, a perturbation function $\msf{pert} : \C^{d_1} \times \cdots \times \C^{d_{n}} \to \C^{d_1} \times \cdots \times \C^{d_{n}}$ will itself have a cascade structure
	\begin{equation}
	\msf{pert}(x_1,\dots,x_n) = \big(\msf{pert}_{1}(x_1),\msf{pert}_{2}(x_1,x_2), \dots, \msf{pert}_{n}(x_1,\dots, x_n) \big).
	\end{equation}

\begin{remark}
By \eqref{eq:linear-cascade-solution} and \eqref{eq:nominal-system-solution}, equations \eqref{eq:asymptotic-equivalence}, \eqref{eq:asymptotic-relative-error}, and \eqref{eq:asymptotic-proportionality} can be rewritten as
	\begin{equation}
	\lim_{t\to\infty} \norm{ \mc L_i^{\circ t}(x_1,\dots,x_i) - L_i^{t}(\msf{pert}_i(x_1,\dots, x_i)) }_{\C^{d_i}} = 0
	\end{equation}

	\begin{equation}
	\lim_{t\to\infty} \frac{\norm{ \mc L_i^{\circ t}(x_1,\dots,x_i) - L_i^{t}(\msf{pert}_i(x_1,\dots, x_i)) }_{\C^{d_i}} }{\norm{L_i}^{t}} = 0
	\end{equation}
and
	\begin{equation}
	\lim_{t\to\infty} \frac{\norm{ \mc L_i^{\circ t}(x_1,\dots,x_i) - L_i^{t}(\msf{pert}_i(x_1,\dots, x_i)) }_{\C^{d_i}} }{\norm{L_i}^{t}} \leq \zeta_i(x_1,\dots,x_{i-1})
	\end{equation}
Clearly, for $\norm{\msf{Nom}_i} < 1$, zero asymptotic relative error implies the other two.
\end{remark}

\subsection{The Koopman operator.}\label{subsec:koopman-operator}
Let $\mc A_i$ be a sub-algebra of $C(\C^{d_i}, \C)$, the set of continuous complex-valued functions on $\C^{d_i}$, where algebra addition is given by normal function addition, $(f+g)(x_i) = f(x_i) + g(x_i)$, and algebra multiplication is given by pointwise multiplication of functions, $(f\cdot g)(x_i) = f(x_i)g(x_i)$. 

Denote the Koopman operator associated with the $i$-th nominal component system \eqref{eq:nominal-linear-system} as 
	\begin{equation}
	\begin{aligned}
	\mc U_{\msf{Nom}_i} : \mc A_i &\to \mc A_i \\
	(\mc U_{\msf{Nom}_i}^{\circ t} f)(x_i) &= f( \msf{Nom}_{i}^{\circ t}(x_i) ) = f( L_{i}^t x_i).
	\end{aligned}
	\end{equation}

\subsubsection{Principal eigenfunctions of component systems.}
In \cite{Mohr:2014wm,Mohr:2016vk}, the principal eigenfunctions of Koopman operator were defined with respect to the linearized system. Assume that $L_i$ is diagonalizable
	\begin{equation}
	L_i = V_i \Lambda_i V_i^{-1},
	\end{equation}
with the columns of $V_i$ being the eigenvectors of $L_i$ and $\Lambda_i = \diag(\lambda_{i,1},\dots, \lambda_{i,d_i})$ being a diagonal matrix containing the eigenvalues of $L_i$. Define the $s$-th principal eigenfunction of the $i$-th system as
	\begin{align}\label{eq:principal-eigenfunction-system-i}
	\psi_{i,s}(x_i) &= (\hat e_{d_i,s}^\ast V_i^{-1}) x_i, &  & (s=1,\dots, d_i)
	\end{align}
where $\hat e_{d_i,s}$ is the $s$-th canonical basis vector of $\C^{d_i}$ and $\hat e_{d_i,s}^{\ast}$ is its conjugate transpose\footnote{Note that $w_{i,s} := (\hat e_{i,s}^\ast V_i^{-1})^{*} = (V_{i}^{-1})^* \hat e_{i,s}$ is the $s$-th dual basis vector in system $i$; that is $\inner{v_{i,t}}{w_{i,s}}_{\C^{d_i}} = w_{i,s}^* v_{i,t} = \delta_{s,t}$, where $v_{i,t}$ is the $t$-th eigenvector of $L_i$}. The function $\psi_{i,s}$ is the $s$-th coordinate functional corresponding to the eigenbasis of the $i$-th system. These are indeed eigenfunctions of $\mc U_{\msf{Nom}_i}$ at eigenvalue $\lambda_{i,s}$ as shown by the following calculation: for any $s\in \set{1,\dots, d_{i}}$
	\begin{equation}
	(\mc U_{\msf{Nom}_i}\psi_{i,s})(x) = \psi_{i,s}(L_i x) = (\hat e_{d_i,s}^\ast V_i^{-1})(L_i x) = \hat e_{d_i,s}^\ast \Lambda_i V_i^{-1} x = \lambda_{i,s} (\hat e_{d_i,s}^\ast V_i^{-1})x = \lambda_{i,s}\psi_{i,s}(x). 
	\end{equation}
For all $i \in \set{1,\dots, n}$, we define $\psi_{i,0}(x) = 1$ so that $\psi_{i,0}$ is an eigenfunction at 1:
	\begin{align}
	(\mc U_{\msf{Nom}_i}\psi_{i,0})(x) = 1 = \psi_{i,0}(x).
	\end{align}

\subsubsection{Koopman operators and principal eigenfunctions for cascaded systems}
We use the algebra of observables $\mc A_i$ on $\C^{d_i}$ to build a space of observables $\mc A$ for the cascade system. Let $\mc A = \mc A_1 \otimes \cdots \otimes \mc A_n$ and denote $\mc U_{\msf{Nom}} : \mc A \to \mc A$ and $\mc U_{\msf{Lin}} : \mc A \to \mc A$ as the Koopman operators associated with the solution operators $\msf{Nom}$ and $\msf{Lin}$, respectively. Recall that a tensor product $f_1 \otimes \cdots \otimes f_n \in \mc A_1 \otimes \cdots \otimes \mc A_n$ acts on multilinear functionals on $\mc A_1 \times \cdots \times \mc A_n$ (see \cite{Ryan:2002ku}). Let $\delta_{(x_1,\dots, x_n)} : \mc A_1 \times \cdots \times \mc A_n \to \C$ be the evaluation functional defined as
	\begin{equation}
	\delta_{(x_1,\dots, x_n)}(f_1, \dots, f_n) = f_1(x_1)\cdots f_n(x_n),
	\end{equation}
where the dots on the right side indicate multiplication in $\C$. The tensor product is defined as
	\begin{align}\label{eq:tensor-product-observables}
	(f_1 \otimes \cdots \otimes f_n)(\delta_{(x_1,\dots, x_n)}) 
	&= \delta_{(x_1,\dots, x_n)}(f_1, \dots, f_n) .
	\end{align}
Due to this, there is no confusion in writing
	\begin{align}
	(f_1 \otimes \cdots \otimes f_n)(x_1,\dots,x_n) \equiv f_1(x_1)\cdots f_n(x_n)
	\end{align}
and we can consider $f_1 \otimes \cdots \otimes f_n$ as a function on the cascaded system's state space $\C^{d_1} \times\cdots\times \C^{d_n}$.

Principal eigenfunctions for $\mc U_{\msf{Nom}} : \mc A \to \mc A$ can be defined from the component systems' principal eigenfunctions \eqref{eq:principal-eigenfunction-system-i} by a trivially extending their domain from the component system's state space $\C^{d_i}$ to the cascaded system's state space $\C^{d_1}\times\cdots\times\C^{d_n}$. To this end, for each $i\in \set{1,\dots,n}$ and $s_i \in \set{1,\dots, d_i}$, define a principal eigenfunction for $\mc U_{\msf{Nom}}$ as
	\begin{align}\label{eq:Nom-principal-eigenfunctions}
	\psi_{(0,\dots,0,s_i,0,\dots,0)}(x_1,\dots,x_n) = (\psi_{i,s_i} \circ \Pi_i)(x_1,\dots,x_n) \equiv \psi_{i,s_i}(x_i).
	\end{align}
Often we will write $\psi_{s_i\hat e_{n,i}} \equiv \psi_{(0,\dots,0,s_i,0,\dots,0)}$. Since by convention $\psi_{i,0} \equiv 1$ for all $i$, this principal eigenfunction for $\mc U_{\msf{Nom}}$ can be written as a tensor product of principal eigenfunctions from the component systems,
	\begin{equation}\label{eq:principal-eigenfunctions-nominal-system-tensor-product-form}
	\begin{aligned}
	\psi_{(0,\dots,0,s_i,0,\dots,0)}(x_1,\dots,x_n) 
	&= \psi_{i,s_i}(x_i) \\
	&\equiv (1\otimes\cdots\otimes 1 \otimes \psi_{i,s_i}\otimes 1 \otimes\cdots\otimes1)(x_1,\dots,x_n) \\
	&= (\psi_{1,0}\otimes\cdots\otimes\psi_{i-1,0}\otimes \psi_{i,s_i}\otimes \psi_{i+1,0}\otimes\cdots\otimes\psi_{n,0})(x_1,\dots,x_n).
	\end{aligned}
	\end{equation}

The multiplication operation in each algebra $\mc A_i$ is given by pointwise products of functions. We can define a multiplication operation for the tensor product, $\bullet : (\mc A_1 \otimes \cdots \otimes \mc A_n) \times (\mc A_1 \otimes \cdots \otimes \mc A_n) \to (\mc A_1 \otimes \cdots \otimes \mc A_n)$, as
	\begin{equation}\label{eq:tensor-product-space-multiplication}
	(f_1 \otimes\cdots\otimes f_n) \bullet (g_1 \otimes\cdots\otimes g_n) 
	= (f_1 \cdot g_1) \otimes\cdots\otimes (f_n \cdot g_n)
	\end{equation}
Products of principal eigenfunctions of the form \eqref{eq:principal-eigenfunctions-nominal-system-tensor-product-form} give eigenfunctions of $\mc U_{\msf{Nom}}$. Denote $\psi_{(s_1,\dots,s_n)} \in \mc A$ as
	\begin{align}\label{eq:principal-eigenfunction-cascade}
	\psi_{(s_1,\dots,s_n)}(x_1,\dots, x_n) = (\psi_{1,s_1} \otimes\cdots\otimes \psi_{n,s_n})(x_1,\dots, x_n).
	\end{align}
where $(s_1,\dots,s_n) \in \set{0,\dots d_1} \times \cdots \times \set{0,\cdots,d_n}$. It is clear that this observable can be constructed as a $\bullet$-product of principal eigenfunctions of the form \eqref{eq:principal-eigenfunctions-nominal-system-tensor-product-form}; namely
	\begin{equation}
	\psi_{(s_1,\dots,s_n)} = \psi_{(s_1,0,\dots,0)} \bullet \psi_{(0,s_2,0,\dots,0)} \bullet \cdots \bullet \psi_{(0,\dots,0,s_n)}.
	\end{equation}
Equation \eqref{eq:principal-eigenfunction-cascade} is an eigenfunction at $\lambda_{1,s_1}\cdots\lambda_{n,s_n}$ as shown by the following computation:
	\begin{align*}
	(\mc U_{\msf{Nom}}\psi_{(s_1,\dots,s_n)})(x_1,\dots, x_n)%
	&= \psi_{(s_1,\dots,s_n)}(\msf{Nom}(x_1,\dots, x_n)) \\%
	&= \psi_{(s_1,\dots,s_n)}(L_1 x_1,\dots, L_n x_n) \\%
	&= (\psi_{1,s_1} \otimes\cdots\otimes \psi_{n,s_n})(L_1 x_1,\dots, L_n x_n) \\%
	&= \psi_{1,s_1}(L_1 x_1)\cdots \psi_{n,s_n}(L_n x_n) \\%
	&= (\mc U_{\msf{Nom}_1}\psi_{1,s_1})(x_1)\cdots (\mc U_{\msf{Nom}_n}\psi_{n,s_n})(x_n) \\%
	&= (\lambda_{1,s_1}\psi_{1,s_1})(x_1)\cdots (\lambda_{n,s_n}\psi_{n,s_n})(x_n) \\%
	&= (\lambda_{1,s_1}\cdots\lambda_{n,s_n})(\psi_{1,s_1}(x_1)\cdots \psi_{n,s_n}(x_n)) \\%
	&= (\lambda_{1,s_1}\cdots\lambda_{n,s_n})(\psi_{1,s_1} \otimes\cdots\otimes \psi_{n,s_n})(x_1,\dots, x_n) \\%
	&= (\lambda_{1,s_1}\cdots\lambda_{n,s_n})\psi_{(s_1,\dots,s_n)}(x_1,\dots, x_n) .
	\end{align*}
The preceding work is an example of the more general result that eigenfunctions of the Koopman operator form a semigroup under pointwise multiplication of functions, as pointed out in \cite{Budisic:2012cf}.

\begin{remark}
Each subalgebra $\mc A_i$ for the component system $i$ is embedded in $\mc A$. The embedding is given by the map
	\begin{equation*}
	f \mapsto \underbrace{1 \otimes\cdots\otimes 1}_{\text{$i-1$ times}} \otimes f \otimes \underbrace{1 \otimes \cdots \otimes 1}_{\text{$n-i$ times}},
	\end{equation*}
where the $f$ on the right hand side is in the $i$-th position. Furthermore, if each $\mc A_i$ is generated by the principal eigenfunctions $\psi_{i,1},\dots, \psi_{i,d_i}$, then $\mc A$ is generated by the principal eigenfunctions $\set{\psi_{(0,\dots,0,s_i,0,\dots,0)} \given \forall i \in\set{1,\dots,n}, \forall s_{i} \in \set{1,\dots, d_i}}$.
\end{remark}

\section{Main results}\label{sec:main-results}
Consider the special case of the nonlinear and linear cascades (eq.'s \eqref{eq:nonlinear-cascade} and \eqref{eq:linear-cascade}, respectively) where system $i$ is only affected by system $i-1$ (see Fig.\ \ref{fig:chained-cascade-system}). This corresponds to the situation where $C_{i,j}$ is non-zero if and only if $j = i-1$ and $N_{i}(x_1,\dots, x_i) = N_{i}(x_{i-1},x_i)$:
	\begin{equation}
	\begin{aligned}\label{eq:nonlinear-chained-cascade}
	x_1(t+1) &= L_1 x_1(t) + N_1(x_1(t))& & &\\
	x_i(t+1) &= L_i x_i(t) + C_{i,i-1} x_{i-1}(t) + N_i(x_{i-1}(t),x_i(t))& &(i=2,\dots, n).
	\end{aligned}
	\end{equation}
and	
	\begin{equation}
	\begin{aligned}\label{eq:linear-chained-cascade}
	x_1(t+1) &= L_1 x_1(t)\\
	x_i(t+1) &= L_i x_i(t) + C_{i,i-1} x_{i-1}(t)& &(i=2,\dots, n).
	\end{aligned}
	\end{equation}
We will call cascades having the form \eqref{eq:nonlinear-chained-cascade} and \eqref{eq:linear-chained-cascade} chained cascades.

\begin{figure}[h]
\begin{center}
\includegraphics[width=0.9\textwidth]{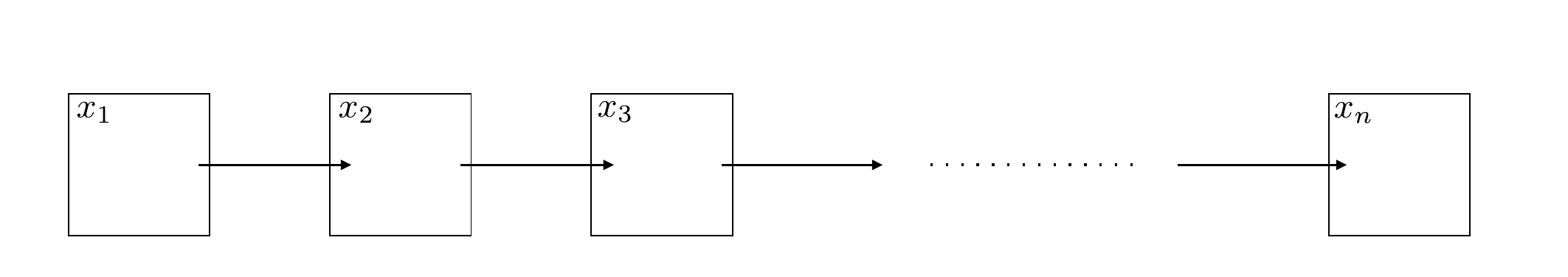}
\caption{Chained cascade system}
\label{fig:chained-cascade-system}
\end{center}
\end{figure}

\begin{condition}\label{eq:chained-cascade-conditions}
The following conditions will be in force for all following results:
\begin{enumerate}[(i)]
\item $L_{i}$ is invertible and diagonalizable for all $i=1,\dots,n$, 
	\begin{equation}
	L_{i} V_{i} = V_{i} \Lambda_{i}.
	\end{equation}
\item (Disjoint spectrums) The spectrums of each layer are pairwise disjoint. That is for $i,j\in\set{1,\dots,n}$ satisfying $i\neq j$ 
	\begin{equation} 
	\sigma(L_{i}) \cap \sigma(L_{j}) = \emptyset.
	\end{equation}
\item\label{cond:hierarchy-of-norms} $\norm{L_1} < \norm{L_2} < \cdots < \norm{L_n} \leq 1$. 
\end{enumerate}
\end{condition}

\begin{remark}
Condition \eqref{cond:hierarchy-of-norms} requires that the response times of upstream systems are faster than the downstream systems they feed into.
\end{remark}

\subsection{Theorems}

\begin{theorem}[0 asymptotic relative error for chained, linear cascades]\label{thrm:linear-asymptotic-equivalence}
Assume Condition \ref{eq:chained-cascade-conditions} is in effect. Then \eqref{eq:linear-chained-cascade} has 0 asymptotically relative error in the sense of \eqref{eq:asymptotic-relative-error}. In particular, for all $i\geq 1$\footnote{We take the empty sum $\sum_{j=1}^{0}$ to be 0.} and all $t\geq 0$,
	\begin{align}
	\norm{ \Pi_{i} \circ \msf{Lin}^{\circ t}(x_1,\dots, x_{n})  -  \Pi_{i} \circ \msf{Nom}^{\circ t}(\msf{pert}(x_1,\dots, x_{n}))}
	&\leq \sum_{j=1}^{i-1} \norm{D_{i,j}}\norm{L_{j}^t \msf{pert}_{j}(x_1,\dots,x_{j})} \label{eq:linear-cascade-absolute-error}\\
	&\leq \left(\sum_{j=1}^{i-1} \norm{D_{i,j}}\norm{\msf{pert}_{j}(x_1,\dots,x_{j})} \right)\norm{L_i}^t \label{eq:thrm-asymptotic-proportionality}
	\end{align}
and for all $i\geq 1$
	\begin{align}\label{eq:thrm-asymptotic-relative-error}
	\lim_{t\to\infty} \frac{\norm{ \Pi_{i} \circ \msf{Lin}^{\circ t}(x_1,\dots, x_{n})  -  \Pi_{i} \circ \msf{Nom}^{\circ t}(\msf{pert}(x_1,\dots, x_{n})) } }{\norm{L_{i}}^{t}}
	= 0.
	\end{align}
	
The perturbation function for the initial condition, $\msf{pert} : \C^{d_1} \times \cdots \times \C^{d_n} \to \C^{d_1} \times \cdots \times \C^{d_n}$, is defined as
	\begin{equation}
	\begin{aligned}\label{eq:perturbation-function}
	\msf{pert}(x_1,\dots, x_n) &= (\msf{pert}_{1}\circ \Pi_1, \msf{pert}_{2}\circ (\Pi_1, \Pi_2),\dots, \msf{pert}_{n-1}\circ (\Pi_1, \dots, \Pi_{n-1}), \msf{pert}_{n} )(x_1,\dots, x_n)
	\end{aligned}
	\end{equation}
and the perturbations for each system $i$,  $\msf{pert}_{i} : \C^{d_1} \times \cdots \times \C^{d_i} \to \C^{d_i}$, are defined recursively by
	\begin{align}
	\msf{pert}_{1}(x_1) &= x_1 \\
	\msf{pert}_{i}(x_1,\dots, x_i) &= x_i + \sum_{j=1}^{i-1} (-1)^{i-1-j} D_{i,j} \msf{pert}_{j}(x_1, \dots, x_{j})  & \forall i\in\set{2,\dots,n}. \label{eq:linear-perturbation-function}
	\end{align}
where
	\begin{align}
	D_{i,i} &= I_{d_i}  & \forall i\in\set{1,\dots,n},\\
	D_{i,j} &= L_{i}^{-1} V_{i} \tilde C_{i,j} V_{j}^{-1} & \forall i\in\set{2,\dots,n}, \forall j\in\set{1,\dots, i-1},
	\end{align}
and the matrix $\tilde C_{i,j} \in \C^{d_{i} \times d_{j}}$ has elements
	\begin{align}
	&[\tilde C_{i,j}]_{\ell,m} = \left[ V_{i}^{-1} C_{i,i-1} D_{i-1,j} V_{j}\right]_{\ell,m} \left(1 - \frac{\lambda_{j,m}}{\lambda_{i,\ell}} \right)^{-1} &
	\forall i\in\set{2,\dots,n}, \forall j\in\set{1,\dots, i-1}.
	\end{align}
\end{theorem}

\begin{remark}
The requirement that the eigenvalues of $L_i$ and $L_{j}$ ($i\neq j$) form disjoint sets is due to the form of the entries of the coupling matrix $\tilde C_{i,j}$. If any pair of eigenvalues from $L_i$ and $L_{j}$ were equal, $\tilde C_{i,j}$ would not be well-defined since its matrix elements have a term of the form $(1 - \lambda_{j,m}/\lambda_{i,\ell})^{-1}$. The requirement of disjoint-ness can be thought of as a non-resonance condition.
\end{remark}

\begin{proof}[Proof of Theorem \ref{thrm:linear-asymptotic-equivalence}]
Equations \eqref{eq:thrm-asymptotic-proportionality} and \eqref{eq:thrm-asymptotic-relative-error} follow from corollary \ref{lem:chained-linear-cascade-asymptotic-proportionality}. The expressions for the perturbation terms $\msf{pert}_{j}$ and the coupling matrices $D_{i,j}$ and $\tilde C_{i,j}$ are derived in the proof of lemma \ref{lemma:level-i-solution-with-perturbation-terms}.
\end{proof}

\begin{corollary}[Asymptotic equivalence for chained, linear cascades]\label{cor:asymptotic-equivalence-cascade-system-absolute-error}
	\begin{align}
	\lim_{t\to\infty} \norm{ \msf{Lin}^{\circ t}(x_1,\dots, x_{n})  -  \msf{Nom}^{\circ t}(\msf{pert}(x_1,\dots, x_{n})) }_{\times} 
	= 0.
	\end{align}
\end{corollary}

\begin{theorem}[Perturbation of principal eigenfunctions]\label{thrm:perturbation-principal-eigenfunctions}
Assume Condition \ref{eq:chained-cascade-conditions} is in effect. For any $i\geq 1$, $s_{i} \in \set{1,\dots, d_i}$, and $t \in \N$
	\begin{align}
	\abs*{\big(\mc U_{\msf{Lin}}^{\circ t} \psi_{(0,\dots,0,s_i,0,\dots,0)}\big)(x_1,\dots,x_n) 
	- \big(\mc U_{\msf{Nom}}^{\circ t}\psi_{(0,\dots,0,s_i,0,\dots,0)}\big)\circ\msf{pert}(x_1,\dots,x_{n}) } \nonumber\\
	\leq \norm{\psi_{i,s_i}} \sum_{j=1}^{i-1} \norm{D_{i,j}}\norm{L_{j}^t \msf{pert}_{j}(x_1,\dots,x_{j})}_{\C^{d_{j}}}. \label{eq:thrm-linear-principal-eigenfunction-proportionality}
	\end{align}
Furthermore, for any $i \in \set{1,\dots,n}$
	\begin{equation}\label{eq:thrm-linear-principal-eigenfunction-0-asymptotic-relative-error}
	\lim_{t\to\infty}\frac{\abs*{\big(\mc U_{\msf{Lin}}^{\circ t} \psi_{(0,\dots,0,s_i,0,\dots,0)}\big)(x_1,\dots,x_n) 
	- \big(\mc U_{\msf{Nom}}^{\circ t}\psi_{(0,\dots,0,s_i,0,\dots,0)}\big)\circ\msf{pert}(x_1,\dots,x_{n}) }}{\norm{L_i}^t}
	 = 0 .
	\end{equation}
\end{theorem}

\begin{proof}
See section \ref{sec:principal-eigenfunction-perturbation-proof} below.
\end{proof}

\begin{remark}
Recall that $\psi_{(0,\dots,0,s_i,0,\dots,0)}(x_1,\dots,x_n) = (\psi_{i,s_i}\circ \Pi_i)(x_1,\dots,x_n)$. Due to the definitions of $\msf{Nom}$ and $\msf{pert}$, 
	\begin{align*}
	\big(\mc U_{\msf{Nom}}^{\circ t}\psi_{(0,\dots,0,s_i,0,\dots,0)}\big)\circ\msf{pert}(x_1,\dots,x_{n})
	&= (\psi_{i,s_i}\circ\Pi_i \circ \msf{Nom}^{\circ t})\circ \msf{pert}(x_1,\dots,x_n) \\
	&= (\psi_{i,s_i}\circ\msf{Nom}_{i}^{\circ t}) \circ\msf{pert}_{i}(x_1,\dots,x_i) \\
	&= (\mc U_{\msf{Nom}_{i}}^{\circ t}\psi_{i,s_i}) \circ \msf{pert}_{i}(x_1,\dots,x_i) \\
	&= (\lambda_{i,s_i}^{t}\psi_{i,s_i}) \circ \msf{pert}_{i}(x_1,\dots,x_i) \\
	&= \lambda_{i,s_i}^{t} \psi_{(0,\dots,0,s_i,0,\dots,0)} \circ\msf{pert}(x_1,\dots,x_{n})
	\end{align*}
\end{remark}

\begin{corollary}\label{cor:gla-eigenfunctions}
Fix $i \in 1,\dots, n$, and let $\lambda_{i,s_i} \in \sigma(L_i)$. Then $\psi_{(0,\dots, 0,s_i,0,\dots,0)} \circ \msf{pert}$ is an eigenfunction of $\mc U_{\msf{Lin}}$ at eigenvalue $\lambda_{i,s_i}$.
\end{corollary}

\begin{proof}
This is a straight forward application of the GLA theorem \cite{Mohr:2014wm}. We only show it for a peripheral eigenvalue ($\abs{\lambda_{i,s_i}} = \norm{L_i}$). This is accomplished by showing that a Laplace average of $\psi_{(0,\dots,0,s_i,0,\dots,0)}$ converges to $\psi_{(0,\dots,0,s_i,0,\dots,0)}\circ \msf{pert}$. 

Let $\hat e_{n,i}$ be the $i$-th canonical basis vector of length $n$ and write $\psi_{s_i\hat e_{n,i}} = \psi_{(0,\dots,0,s_i,0,\dots,0)}$. Form the Laplace average,
	\begin{align*}
	\frac{1}{N} \sum_{t=0}^{N-1} \lambda_{i,s_i}^{-t} \mc U_{\msf{Lin}}^{\circ t}\psi_{s_i\hat e_{n,i}} 
	&= \frac{1}{N} \sum_{t=0}^{N-1} \lambda_{i,s_i}^{-t} \left[ (\mc U_{\msf{Nom}}^{\circ t}\psi_{s_i\hat e_{n,i}}) \circ \msf{pert} + \mc U_{\msf{Lin}}^{\circ t}\psi_{s_i\hat e_{n,i}} - (\mc U_{\msf{Nom}}^{\circ t}\psi_{s_i\hat e_{n,i}}) \circ \msf{pert} \right] \\
	&= \frac{1}{N} \sum_{t=0}^{N-1} (\psi_{s_i\hat e_{n,i}}\circ \msf{pert}) + \lambda_{i,s_i}^{-t} \left( \mc U_{\msf{Lin}}^{\circ t}\psi_{s_i\hat e_{n,i}} - (\mc U_{\msf{Nom}}^{\circ t}\psi_{s_i\hat e_{n,i}}) \circ \msf{pert}\right),
	\end{align*}
where we have used $\mc U_{\msf{Nom}}^{\circ t}\psi_{s_i\hat e_{n,i}} = \mc \lambda_{i,s_i}^t\psi_{s_i\hat e_{n,i}}$. Then 
	\begin{align*}
	\norm*{\frac{1}{N} \sum_{t=0}^{N-1} \lambda_{i,s_i}^{-t} \mc U_{\msf{Lin}}^{\circ t}\psi_{s_i\hat e_{n,i}} - \psi_{s_i\hat e_{n,i}}\circ \msf{pert} }
	&\leq \frac{1}{N} \sum_{t=0}^{N-1} \norm*{\lambda_{i,s_i}^{-t} \left( \mc U_{\msf{Lin}}^{\circ t}\psi_{s_i\hat e_{n,i}} - (\mc U_{\msf{Nom}}^{\circ t}\psi_{s_i\hat e_{n,i}}) \circ \msf{pert}\right) } \\
	&= \frac{1}{N} \sum_{t=0}^{N-1} \frac{\norm*{\left( \mc U_{\msf{Lin}}^{\circ t}\psi_{s_i\hat e_{n,i}} - (\mc U_{\msf{Nom}}^{\circ t}\psi_{s_i\hat e_{n,i}}) \circ \msf{pert}\right)} }{\norm{L_i}^t}.
	\end{align*}
It is clear from \eqref{eq:thrm-linear-principal-eigenfunction-0-asymptotic-relative-error}, that the right-hand side converges to 0 as $N\to\infty$.

In the general case, to project onto the $\lambda$ eigenspace, $\mc U_{\msf{Lin}}$ in the above average is replaced with $\mc U_{\msf{Lin}} (I - P)$ where $P$ is the projection onto the direct sum of $\mu$-eigenspace, for $\mu$ satisfying $\abs{\mu} > \abs{\lambda}$.
\end{proof}

\begin{remark}
It can be shown that $\psi_{s_i\hat e_{n,i}}\circ \msf{pert}$ is an eigenfunction of $\mc U_{\msf{Lin}}$, without appeal to the GLA theorem, by direct computation, but it is more involved since $\msf{pert}$ for an $n$-layer cascade consists of a product of $n-1$ lower block triangular matrices. Example \ref{ex:2-layer-eigenfunction-computation} below shows the computation for just a two layer system.
\end{remark}

\begin{example}\label{ex:2-layer-eigenfunction-computation}
We demonstrate that the result of corollary \ref{cor:gla-eigenfunctions} explicitly for a 2 layer chained cascade;
	\begin{align*}
	x_1(t+1) &= L_1 x_1(t) \\
	x_2(t+1) &= L_2 x_2(t) + C_{2,1} x_1(t) .
	\end{align*}
Using theorem \ref{thrm:linear-asymptotic-equivalence},
	\begin{align*}
	\msf{pert}_1(x_1,x_2) &= \begin{bmatrix} I_{d_1} & \mbf 0\end{bmatrix}\begin{bmatrix} x_1 \\ x_2 \end{bmatrix} \\
	\msf{pert}_2(x_1,x_2) &= \begin{bmatrix} D_{2,1} &  I_{d_2} \end{bmatrix}\begin{bmatrix} x_1 \\ x_2 \end{bmatrix} 
	\end{align*}
Therefore,
	\begin{equation}
	\msf{pert}(x_1,x_2) 	
	=
		\begin{bmatrix} 
		\msf{pert}_1(x_1) \\
		\msf{pert}_2(x_1,x_2)
		\end{bmatrix}
	=
    	\begin{bmatrix} 
    		I_{d_1} & \mbf 0 \\
    		D_{2,1} &  I_{d_2}
    	\end{bmatrix}
		\begin{bmatrix} x_1 \\ x_2 \end{bmatrix}
	\end{equation}
Furthermore, we have the principal eigenfunction for the second system
	\begin{equation}
	\psi_{(0,s_2)}(x_1,x_2) = 
	\begin{bmatrix} 
		0 & \hat e_{d_2,s_2}^\ast V_{2}^{-1}
	\end{bmatrix}
	\begin{bmatrix} x_1 \\ x_2 \end{bmatrix}.
	\end{equation}
By lemma \ref{lemma:level-i-solution-with-perturbation-terms}
	\begin{equation}
	\msf{Lin}^{\circ t}(x_1,x_2) =
	\begin{bmatrix}
		L_1^t \msf{pert}_1(x_1) \\
		L_2^t \msf{pert}_2(x_1,x_2) - D_{2,1}L_1^t \msf{pert}_1(x_1)
	\end{bmatrix}.
	\end{equation}
We can write $\msf{Lin}^{\circ t}(x_1,x_2)$ as
	\begin{align*}
	\msf{Lin}^{\circ t}(x_1,x_2) &=
	\begin{bmatrix}
		I_{d_1} & \mbf 0 \\
		- D_{2,1} & I_{d_2}
	\end{bmatrix}
	\begin{bmatrix}
		L_1^t &  \\
		  & L_2^t
	\end{bmatrix}
	\begin{bmatrix} 
		\msf{pert}_{1}(x_1) \\
		\msf{pert}_{2}(x_2)
	\end{bmatrix}.
	\end{align*}
Our goal is to show that $\psi_{(0,s_2)} \circ \msf{pert}$ is an eigenfunction at eigenvalue $\lambda_{2,s_2}$ for $\mc{U}_{\msf{Lin}}^{\circ t}$. To this end we compute
	\begin{align*}
	\mc U_{\msf{Lin}}^{\circ t}(\psi_{(0,s_2)} \circ \msf{pert})(x_1,x_2)
	&= \psi_{(0,s_2)} \circ \msf{pert} \circ \msf{Lin}^{\circ t}(x_1,x_2) \\
	&= 	\underbrace{\begin{bmatrix} 
			0 & \hat e_{d_2,s_2}^\ast V_{2}^{-1}
		\end{bmatrix}}_{\psi_{(0,s_2)}}
		\underbrace{\begin{bmatrix} 
			I_{d_1} & \mbf 0 \\
			D_{2,1} &  I_{d_2}
		\end{bmatrix}}_{\msf{pert}}
		\underbrace{\begin{bmatrix}
		I_{d_1} & \mbf 0 \\
		- D_{2,1} & I_{d_2}
    	\end{bmatrix}
    	\begin{bmatrix}
    		L_1^t &  \\
    		  & L_2^t
    	\end{bmatrix}
    	\begin{bmatrix} 
    		\msf{pert}_{1}(x_1) \\
    		\msf{pert}_{2}(x_2)
    	\end{bmatrix}}_{\msf{Lin}^{\circ t}(x_1,x_2)}
	\\
	&= 	\begin{bmatrix} 
			0 & \hat e_{d_2,s_2}^\ast V_{2}^{-1}
		\end{bmatrix}
    	\begin{bmatrix}
    		L_1^t &  \\
    		  & L_2^t
    	\end{bmatrix}
    	\begin{bmatrix} 
    		\msf{pert}_{1}(x_1) \\
    		\msf{pert}_{2}(x_2)
    	\end{bmatrix}
	\\
	&= 	\hat e_{d_2,s_2}^\ast V_{2}^{-1}
		L_2^t
		\msf{pert}_{2}(x_2)
	\\
	&= 	\hat e_{d_2,s_2}^\ast  \Lambda_2^t
		V_2^{-1}
		\msf{pert}_{2}(x_2)
	\\
	&= 	\lambda_{2,s_2}^{t}\hat e_{d_2,s_2}^\ast 
		V_2^{-1}
		\msf{pert}_{2}(x_2)
	\\
	&= 	\lambda_{2,s_2}^t
		\begin{bmatrix} 
			0 & \hat e_{d_2,s_2}^\ast V_{2}^{-1}
		\end{bmatrix}
    	\begin{bmatrix} 
    		\msf{pert}_{1}(x_1) \\
    		\msf{pert}_{2}(x_2)
    	\end{bmatrix}
	\\
	&= \lambda_{2,s_2}^t (\psi_{(0,s_2)} \circ \msf{pert})(x_1,x_2).
	\end{align*}
This completes the example.
\qed\end{example}

The asymptotic equivalence for the chained, linear cascades can be pushed to chained, nonlinear cascades with asymptotically stable fixed points through the use of a topological conjugacy. Let $\tau = (\tau_1,\dots,\tau_n) : \C^{d_1}\times \cdots \times \C^{d_n} \to \C^{d_1}\times \cdots \times \C^{d_n}$ be a topological conjugacy from the linear system to the nonlinear system and $\tau^{-1} = (\rho_1, \dots, \rho_n) : \C^{d_1}\times \cdots \times \C^{d_n} \to \C^{d_1}\times \cdots \times \C^{d_n}$ its inverse. In general, both $\tau_i$ and $\rho_i$ are maps from $\C^{d_1}\times \cdots \times \C^{d_n} \to \C^{d_i}$. The topological conjugacy makes the following diagram commute:

\begin{equation}
\begin{tikzcd}
\C^{d_1}\times\cdots\times\C^{d_n} \arrow{d}{\tau} \arrow{r}{\msf{Lin}^{\circ t}} & \C^{d_1}\times\cdots\times\C^{d_n} \arrow{d}{\tau} \\
\C^{d_1}\times\cdots\times\C^{d_n} \arrow{r}{\msf{NonLin}^{\circ t}} & \C^{d_1}\times\cdots\times\C^{d_n}
\end{tikzcd}
\end{equation}

\begin{theorem}[Asymptotic equivalence for nonlinear cascaded systems]\label{thrm:nonlinear-chained-cascade-asymptotic-equivalence}
Let the conditions of Theorem \ref{thrm:linear-asymptotic-equivalence} be satisfied and let $\tau = (\tau_1,\dots, \tau_n) : \C^{d_1}\times \cdots \times \C^{d_n} \to \C^{d_1}\times \cdots \times \C^{d_n}$ be a topological conjugacy satisfying $\msf{Lin} = \tau^{-1} \circ \msf{NonLin} \circ \tau$. Then, for each initial condition $(y_1,\dots, y_n)$ for the nonlinear system, $\msf{NonLin}$ is asymptotically equivalent to $\tau \circ \msf{Nom} \circ \tau^{-1}$ with the perturbation function $\tau \circ \msf{pert}\circ \tau^{-1}$:
	\begin{align}\label{eq:nonlinear-asymptotic-equivalence}
	\lim_{t\to\infty}\norm[\Big]{ \msf{NonLin}^{\circ t}(y_1,\dots,y_n) - \left( \tau \circ \msf{Nom} \circ \tau^{-1}\right)^{\circ t}(\tau \circ \msf{pert}\circ \tau^{-1})(y_1,\dots,y_n)  } = 0,
	\end{align}
where $\msf{pert}$ is given by \eqref{eq:perturbation-function}.
\end{theorem}

\begin{proof}
See section \ref{sec:nonlinear-chained-cascade-asymptotic-equivalence-proof} for the proof.
\end{proof}

\begin{remark}
Recall that $\msf{Nom}$ is the nominal decoupled linear system and $\tau$ is a map from the linear to the nonlinear system. Then $\tau \circ \msf{Nom} \circ \tau^{-1}$ is a map on the same state space as the nonlinear system $\msf{NonLin}$ and can be thought of as the nominal nonlinear system. Furthermore, since $\msf{pert}$ is the perturbation function for the initial conditions of the linear system, then $\tau \circ \msf{pert}\circ \tau^{-1}$ is the perturbation function for the nonlinear system's initial conditions.
\end{remark}

\begin{theorem}[Perturbation of eigenfunctions for nonlinear cascades]\label{thrm:peturbation-nonlinear-principal-eigenfunctions}
Let the conditions of Theorem \ref{thrm:nonlinear-chained-cascade-asymptotic-equivalence} be satisfied. For any $\vec y = (y_1,\dots,y_n) \in \C^{d_1}\times\cdots\times\C^{d_n}$,
	\begin{align}
	\lim_{t\to\infty} \frac{\abs*{ \mc U_{\msf{NonLin}}^{\circ t} (\psi_{(0,\dots,0,s_i,0,\dots, 0)} \circ \tau^{-1})(\vec y) - \mc U_{\tau\circ\msf{Nom}\circ\tau^{-1}}^{\circ t} (\psi_{(0,\dots,0,s_i,0,\dots, 0)} \circ \tau^{-1})((\tau\circ\msf{pert}\circ\tau^{-1})(\vec y)) }}{\norm{L_i}^t} = 0.
	\end{align}
\end{theorem}

\begin{proof}
See section \ref{sec:perturbation-nonlinear-principal-eigenfunctions-proof}.
\end{proof}

\begin{remark}
It was shown in \cite{Budisic:2012cf} that if $\psi$ was an eigenfunction corresponding to the Koopman operator associated with the linearized system and $\tau$ was a topological conjugacy from the linear to the nonlinear system, then $\psi \circ \tau^{-1}$ was an eigenfunction of the Koopman operator associated with the nonlinear system. Recall that $\psi_{(0,\dots,0,s_i,0,\dots, 0)}$ is a principal eigenfunction for the Koopman operator $\mc U_{\msf{Nom}}$ associated with the nominal linear system. Then $\psi_{(0,\dots,0,s_i,0,\dots, 0)} \circ \tau^{-1}$ is a principal eigenfunction for the Koopman operator $\mc U_{\tau\circ\msf{Nom}\circ\tau^{-1}}$ associated with the nonlinear nominal system $\tau\circ\msf{Nom}\circ\tau^{-1}$.
\end{remark}

\begin{remark}
Theorem \ref{thrm:peturbation-nonlinear-principal-eigenfunctions} says that the action of the Koopman operator associated with the nonlinear cascade $\msf{NonLin}$ on the observable $\psi_{(0,\dots,0,s_i,0,\dots, 0)} \circ \tau^{-1}$ is asymptotically equivalent to the action of the Koopman operator associated with $\tau\circ\msf{Nom}\circ\tau^{-1}$ (the nominal nonlinear system) on $\psi_{(0,\dots,0,s_i,0,\dots, 0)} \circ \tau^{-1}$ but at a perturbed initial condition, but when evaluated at a perturbed initial condition $(\tau\circ\msf{pert}\circ\tau^{-1})(\vec y)$. 
\end{remark}

\begin{remark}
While the preceding results are proved for chained cascades, they should be easily extensible to the general cascade systems. The difference should only be in the exact form of the perturbation functions and the bounds. The asymptotic results should remain the same.
\end{remark}

\subsection{Numerical experiments.}
The results of theorem \ref{thrm:linear-asymptotic-equivalence} were confirmed with simulation. The simulation consisted of 7-layer\footnote{The choice of 7 layers was made merely for display purposes. Results for larger cascades have been confirmed as well.} linear chained, cascaded system with randomly generated dimensions for each system $i$. System matrices $L_i$ were randomly generated with entries uniformly in the interval $[-1,1]$ and then scaled to have $\norm{L_i} = (0.9)^{8-i}$ for $i=1,\dots, 7$. The coupling matrices $C_{i,i-1}$ were also randomly generated with entries uniformly in $[-1,1]$. Initial conditions for each system $i$ were randomly generated and scaled to have $\norm{x_i} = 1$.

Figures \ref{fig:linear-cascade-absolute-error} and \ref{fig:linear-cascade-relative-error} show the log absolute error and log relative errors of a typical run of the simulation. The black asterisks in fig. \ref{fig:linear-cascade-absolute-error} are the predicted upper bound \eqref{eq:thrm-asymptotic-proportionality}. The colored lines in each plot correspond to the log of the absolute error, $\log\left(\norm{\Pi_i\circ \msf{Lin}^{\circ t}(x_1,\dots, x_n) - \Pi_i\circ \msf{Nom}^{\circ t}(\msf{pert}(x_1,\dots, x_n)) }\right)$, between the linear and the nominal systems. Figure \ref{fig:linear-cascade-relative-error} shows the log of the relative error
	\begin{equation}\label{eq:log-relative-error}
	\log\left( \frac{\norm{\Pi_i\circ \msf{Lin}^{\circ t}(x_1,\dots, x_n) - \Pi_i\circ \msf{Nom}^{\circ t}(\msf{pert}(x_1,\dots, x_n)) }}{\norm{L_i}^t} \right).
	\end{equation}
As can be seen, the log absolute error and log relative error decrease linearly, confirming that the absolute error and relative error decrease exponentially fast to zero.

\begin{figure}[ht]
\begin{center}
\includegraphics[width=1\textwidth]{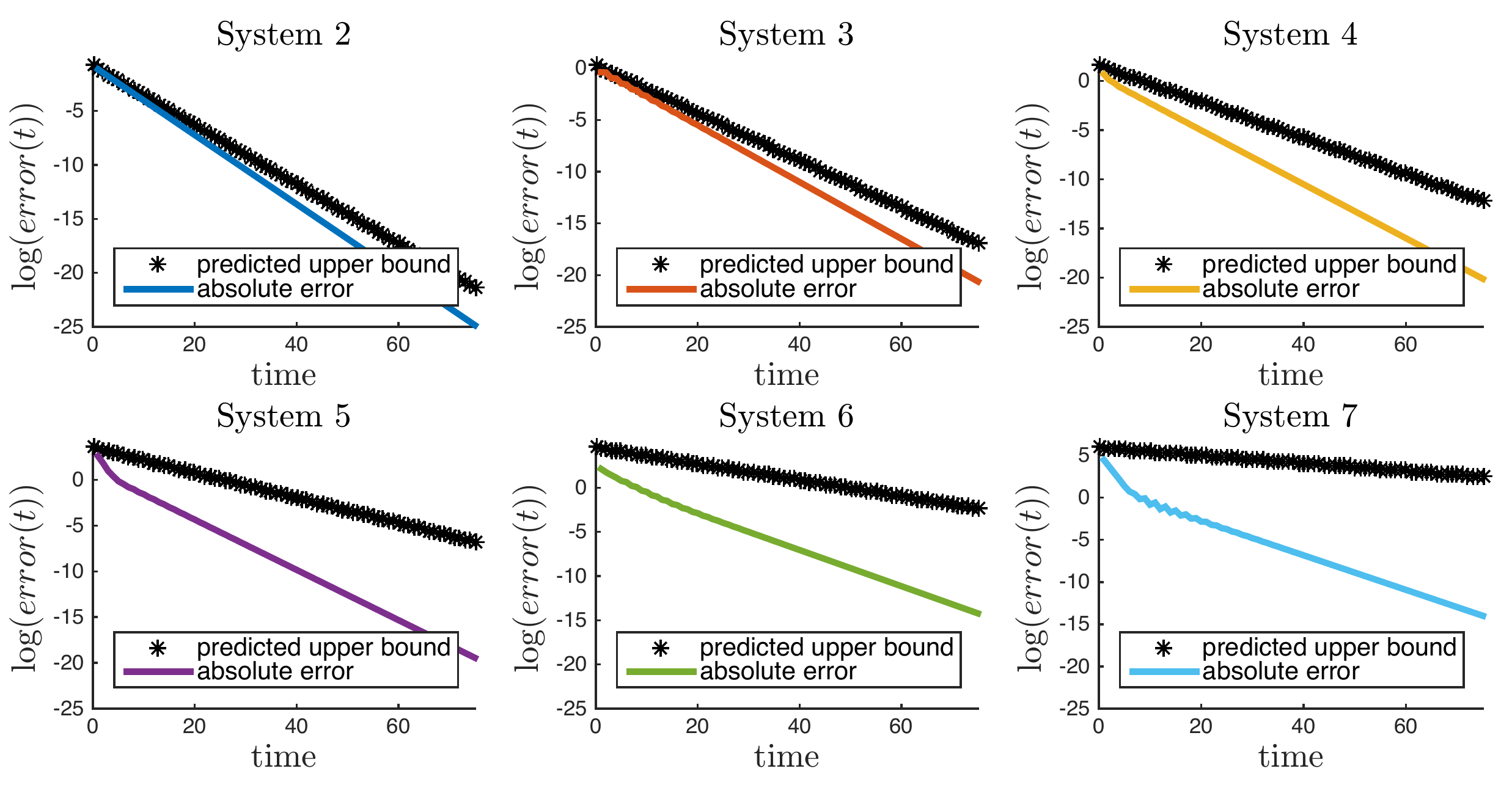}
\caption{Log absolute error between $\Pi_i\circ\msf{Lin}^{\circ t}$ and $\Pi_i\circ\msf{Nom}^{\circ t}$. The black asterisks ($\ast$) correspond to the upper bound given by \eqref{eq:thrm-asymptotic-proportionality}. The colored lines correspond to the absolute error given by the right hand side of \eqref{eq:linear-cascade-absolute-error}. Traces for system 1 are not plotted since by construction $\Pi_1\circ \msf{Lin}^{\circ t} = \Pi_1 \circ \msf{Nom}^{\circ t}$ for all $t$.}
\label{fig:linear-cascade-absolute-error}
\end{center}
\end{figure}

\begin{figure}[ht]
\begin{center}
\includegraphics[width=0.6\textwidth]{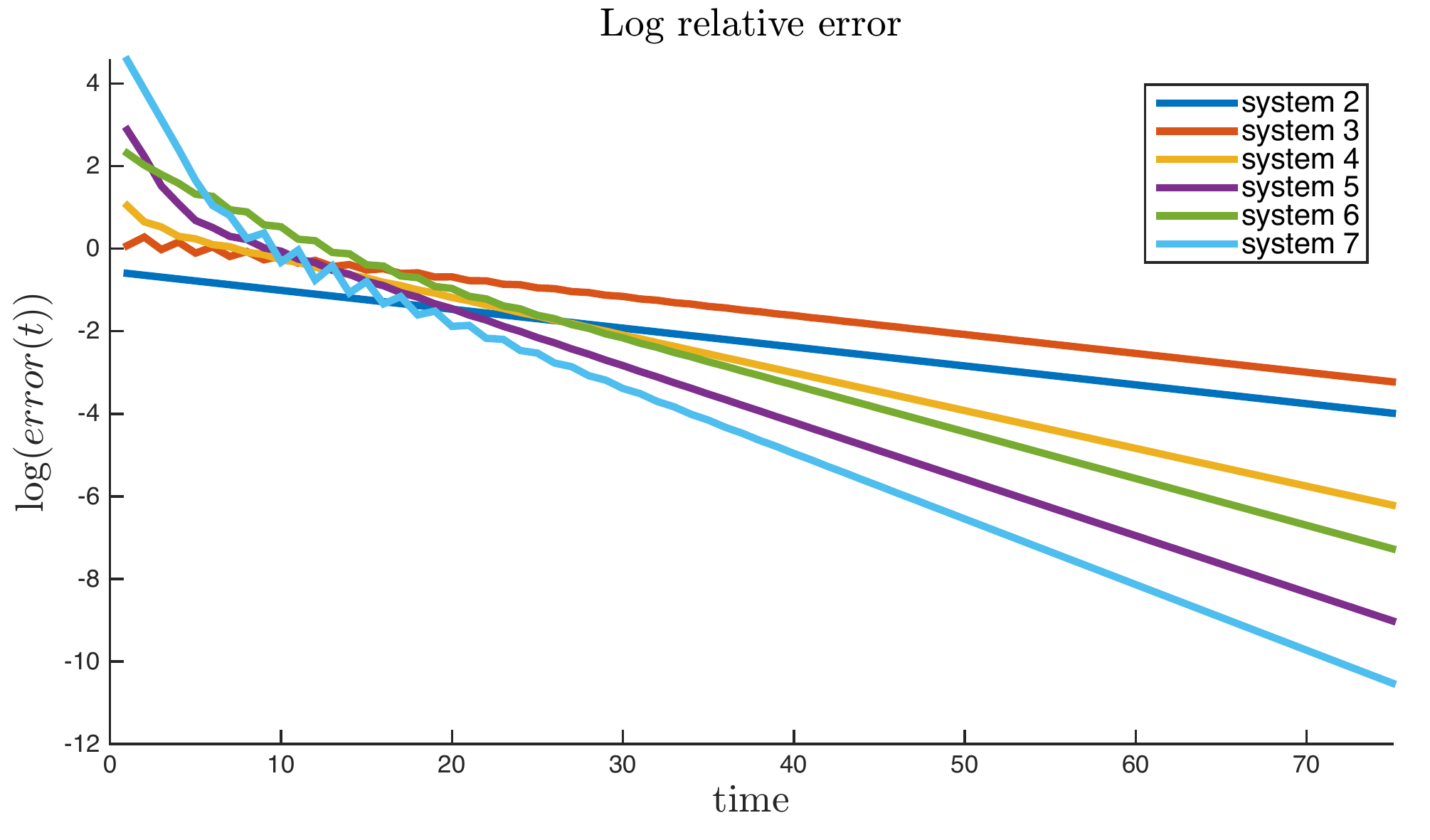}
\caption{Log relative error between $\Pi_i\circ\msf{Lin}^{\circ t}(x_1,\dots,x_n)$ and $\Pi_i\circ\msf{Nom}^{\circ t}(\msf{pert}(x_1,\dots, x_n))$. The colored lines correspond to the log relative error \eqref{eq:log-relative-error}.}
\label{fig:linear-cascade-relative-error}
\end{center}
\end{figure}
%
\section{Proofs of the main results}\label{sec:proofs}

\subsection{Asymptotic equivalence of linear, chained cascades.}
The first lemma gives the general solution for the $i$-th level of the chained linear cascade system
	
\begin{lemma}\label{lemma:general-solution}
For all $i=1,\dots, n$ and $t \geq 0$, denote by $x_i(t)$ the solution $\Pi_i \circ \msf{Lin}^{\circ t}(x_1,\dots,x_n)$ of the $i$-th level of \eqref{eq:linear-chained-cascade}. For $i\geq 2$, the general solution satisfies
	\begin{equation}\label{eq:linear-chained-cascade-general-solution}
	x_{i}(t) \equiv \Pi_i \circ \msf{Lin}^{\circ t}(x_1,\dots, x_n) = L_{i}^{t} x_i + L_{i}^{t-1} V_{i} \sum_{k=0}^{t-1} \Lambda_{i}^{-k} V_{i}^{-1} C_{i,i-1} x_{i-1}(k).
	\end{equation}
\end{lemma}	

\begin{proof}
Repeatedly using \eqref{eq:linear-chained-cascade}, we have
	\begin{align*}
	x_{i}(t) 
	&= L_{i} x_{i}(t-1) + C_{i,i-1}x_{i-1}(t-1) \\
	&= L_{i} \left[ L_{i} x_{i}(t-2) + C_{i,i-1}x_{i-1}(t-2) \right] + C_{i,i-1}x_{i-1}(t-1) \\
	&= L_{i}^{2} x_{i}(t-2) + \left[ L_{i} C_{i,i-1}x_{i-1}(t-2) + C_{i,i-1}x_{i-1}(t-1) \right] \\
	&~~\vdots \\
	&= L_{i}^{t} x_{i}(0) + \left[ L_{i}^{t-1} C_{i,i-1} x_{i-1}(0) + L_{i}^{t-2} C_{i,i-1} x_{i-1}(1) + \cdots L_{i}^{1} C_{i,i-1} x_{t-2}(1) + C_{i,i+1} x_{i-1}(t-1)\right] \\
	&= L_{i}^{t} x_{i}(0) + \sum_{k=0}^{t-1} L_{i}^{t-1-k} C_{i,i-1} x_{i-1}(k) \\ 
	&= L_{i}^{t} x_{i}(0) + L_{i}^{t-1}\sum_{k=0}^{t-1} L_{i}^{-k} C_{i,i-1} x_{i-1}(k) .
	\end{align*}
Replacing $L_{i}^{-k}$ with $V_{i} \Lambda_{i}^{-k} V_{i}^{-1}$ in this final expression gives \eqref{eq:linear-chained-cascade-general-solution}.
\end{proof}
	
\begin{lemma}\label{lemma:perturbation-terms}
Assume Condition \ref{eq:chained-cascade-conditions} holds for \eqref{eq:linear-chained-cascade} and each $L_i$ is diagonalized by $L_{i} = V_{i} \Lambda_{i} V_{i}^{-1}$. For any matrix $B \in \C^{d_i \times d_j}$, the following equality holds for any $i,j \in \set{1,\dots, n}$ with $i\neq j$:
	\begin{align}\label{eq:perturbation-evolution}
	 \sum_{k=0}^{t-1} \Lambda_{i}^{-k} B \Lambda_{j}^{k} 
	= \tilde B - \Lambda_{i}^{-t} \tilde B \Lambda_{j}^{t}
	\end{align}
where $\tilde B \in \C^{d_i \times d_j}$ is the matrix whose $(\ell,m)$-th entry is given by
	\begin{equation}
	[\tilde B]_{\ell,m} = [ B ]_{\ell,m} \left(1 - \frac{\lambda_{j,m} }{\lambda_{i,\ell}}\right)^{-1}
	\end{equation}
\end{lemma}

\begin{proof}
For any matrix $M$ we denote the $(\ell,m)$-th entry as $[M]_{\ell,m}$. The $(\ell,m)$-th entry of \eqref{eq:perturbation-evolution} is given by 
	\begin{align*}
	[\Lambda_{i}^{-k} B \Lambda_{j}^{k}]_{\ell,m}  
	&= \sum_{s=1}^{d_i} [\Lambda_{i}^{-k}]_{\ell,s} [B \Lambda_{j}^{k}]_{s,m} \\
	&= \sum_{s=1}^{d_i} [\Lambda_{i}^{-k}]_{\ell,s} \sum_{u=1}^{d_j} [B]_{s,u}[\Lambda_{j}^{k}]_{u,m} 
	\end{align*}
Since $\Lambda_i$ is diagonal, $[\Lambda_{j}^{k}]_{u,m} = 0$ for $u\neq m$ and $[\Lambda_{j}^{k}]_{m,m} = \lambda_{j,m}^{k}$. This gives
	\begin{align*}
	[\Lambda_{i}^{-k} B \Lambda_{j}^{k}]_{\ell,m}  
	&= \sum_{s=1}^{d_i} [\Lambda_{i}^{-k}]_{\ell,s} [B]_{s,m} \lambda_{j,m}^{k}
	\end{align*}
Since $\Lambda_{i}^{-k}$ is diagonal, we have
	\begin{equation}\label{eq:lambda-b-lambda}
	[\Lambda_{i}^{-k} B \Lambda_{j}^{k}]_{\ell,m} = \lambda_{i,\ell}^{-k} [B]_{\ell,m} \lambda_{j,m}^k 
	= [B]_{\ell,m} \left(\frac{\lambda_{j,m}}{\lambda_{i,\ell}}\right)^k
	\end{equation}
Summing from $k=0,\dots, t-1$, gives
	\begin{align*}
	\sum_{k=0}^{t-1} [\Lambda_{i}^{-k} B \Lambda_{j}^{k}]_{\ell,m} 
	= \sum_{k=0}^{t-1} [B]_{\ell,m} \left(\frac{\lambda_{j,m}}{\lambda_{i,\ell}}\right)^k
	= [B]_{\ell,m} \frac{1 - \left(\frac{\lambda_{j,m}}{\lambda_{i,\ell}}\right)^t}{1- \left(\frac{\lambda_{j,m}}{\lambda_{i,\ell}}\right)}
	= [\tilde B]_{\ell,m} - [\tilde B]_{\ell,m} \left(\frac{\lambda_{j,m}}{\lambda_{i,\ell}}\right)^t
	\end{align*}
Using \eqref{eq:lambda-b-lambda}, but with $B$ and $k$ replaced by $\tilde B$ and $t$, respectively, we get
	\begin{equation*}
	[\tilde B]_{\ell,m} \left(\frac{\lambda_{j,m}}{\lambda_{i,\ell}}\right)^t = [\Lambda_{i}^{-t} \tilde B \Lambda_{j}^{t}]_{\ell,m}.
	\end{equation*}
Therefore,
	\begin{align*}
	\left[\Big( \sum_{k=0}^{t-1} \Lambda_{i}^{-k} B \Lambda_{j}^{k} \Big)\right]_{\ell,m} 
	&= \sum_{k=0}^{t-1} [\Lambda_{i}^{-k} B \Lambda_{j}^{k}]_{\ell,m} \\
	&= [\tilde B]_{\ell,m} - [\tilde B]_{\ell,m} \left(\frac{\lambda_{j,m}}{\lambda_{i,\ell}}\right)^t \\
	&= [\tilde B]_{\ell,m} - [\Lambda_{i}^{-t} \tilde B \Lambda_{j}^{t}]_{\ell,m} \\
	&= [\tilde B  -\Lambda_{i}^{-t} \tilde B \Lambda_{j}^{t}]_{\ell,m} 
	\end{align*}
This is equivalent to \eqref{eq:perturbation-evolution}.
\end{proof}

\begin{lemma}\label{lemma:level-i-solution-with-perturbation-terms}
For each $i=2,\dots,n$, the solution of \eqref{eq:linear-chained-cascade} is 
	\begin{equation}\label{eq:i-th-system-solution-linear-cascade}
	\Pi_i \circ \msf{Lin}^{\circ t}(x_1,\dots,x_n) 
	= \sum_{j=1}^{i} (-1)^{i-j} D_{i,j} L_{j}^{t} \msf{pert}_{j}(x_1,\dots, x_{j}),
	\end{equation}
where
	\begin{align}
	D_{i,i} &= I_{d_i}  & \forall i\in\set{1,\dots,n},\\
	D_{i,j} &= L_{i}^{-1} V_{i} \tilde C_{i,j} V_{j}^{-1} & \forall i\in\set{2,\dots,n}, \forall j\in\set{1,\dots, i-1}, \label{eq:D-i-iminusj}
	\end{align}
and the matrix $\tilde C_{i,j} \in \C^{d_{i} \times d_{j}}$ has elements
	\begin{align}\label{eq:tilde-C-ell-m-elements}
	&[\tilde C_{i,j}]_{\ell,m} = \left[ V_{i}^{-1} C_{i,i-1} D_{i-1,j} V_{j}\right]_{\ell,m} \left(1 - \frac{\lambda_{j,m}}{\lambda_{i,\ell}} \right)^{-1} &
	\forall i\in\set{2,\dots,n}, \forall j\in\set{1,\dots, i-1}.
	\end{align}
The perturbation functions $\msf{pert}_{i} : \C^{d_1} \times \cdots \times \C^{d_i} \to \C^{d_i}$ are multilinear maps defined inductively by
	\begin{align}
	\msf{pert}_{1}(x_1) &= x_1 \\
	\msf{pert}_{i}(x_1,\dots, x_i) &= x_i + \sum_{j=1}^{i-1} (-1)^{i-1-j} D_{i,j} \msf{pert}_{j}(x_1, \dots, x_{j})  & \forall i\in\set{2,\dots,n}. \label{eq:pert-iminusj}
	\end{align}
\end{lemma}

\begin{proof}
We prove the result using induction. First note that the solution for $\Pi_1\circ \msf{Lin}^{\circ t}(x_1,\dots,x_n)$ can be written as
	\begin{equation}\label{eq:linear-chained-cascade-x1-solution}
	x_1(t) = \Pi_1\circ \msf{Lin}^{\circ t}(x_1,\dots,x_n) = L_{1}^{t} x_1 \equiv D_{1,1} L_{1}^{t} \msf{pert}_1(x_1).
	\end{equation}

\begin{enumerate}[(I)]
\item[(\textbf{Seed step}):] Consider $x_2(t) = \Pi_2 \circ \msf{Lin}^{\circ t}(x_1,\dots,x_n)$. By lemma \ref{lemma:general-solution}, eq. \eqref{eq:linear-chained-cascade-general-solution}, this is
	\begin{align}
	x_2(t) &= L_{2}^{t} x_2 + L_{2}^{t-1} V_{2} \sum_{k=0}^{t-1} \Lambda_{2}^{-k} V_{2}^{-1} C_{2,1} x_{1}(k) \nonumber \\
	&= L_{2}^{t} x_2 + L_{2}^{t-1} V_{2} \sum_{k=0}^{t-1} \Lambda_{2}^{-k} V_{2}^{-1} C_{2,1} D_{1,1} L_{1}^{k} \msf{pert}_1(x_1)
	\end{align}
where in the second line we have replaced $x_1(k)$ with \eqref{eq:linear-chained-cascade-x1-solution} for $t = k$. Using $L_{1}^{k} = V_{1} \Lambda_{1}^{k}V_{1}^{-1}$ in the second linear gives
	\begin{equation}\label{eq:x2-solution-1}
	x_2(t) = L_{2}^{t} x_2 + L_{2}^{t-1} V_{2} \left(\sum_{k=0}^{t-1} \Lambda_{2}^{-k} V_{2}^{-1} C_{2,1} D_{1,1} V_{1} \Lambda_{1}^{k}\right)V_{1}^{-1} \msf{pert}_1(x_1)
	\end{equation}
By lemma \ref{lemma:perturbation-terms}, \eqref{eq:perturbation-evolution}, with $B \equiv V_{2}^{-1} C_{2,1} D_{1,1} V_{1}$ gives that
	\begin{equation}\label{eq:perturbation-summation-2-1}
	\left(\sum_{k=0}^{t-1} \Lambda_{2}^{-k} V_{2}^{-1} C_{2,1} D_{1,1} V_{1} \Lambda_{1}^{k}\right) = \tilde C_{2,1} - \Lambda_{2}^{-t} \tilde C_{2,1} \Lambda_{1}^{t}
	\end{equation}
where the elements of $\tilde C_{2,1}$ are given as
	\begin{align}\label{eq:tilde-C-2-1}
	&[\tilde C_{2,1}]_{\ell,m} = [ V_{2}^{-1} C_{2,1} D_{1,1} V_{1}  ]_{\ell,m} \Big(1 - \frac{\lambda_{1,m}}{\lambda_{2,\ell}} \Big)^{-1},
	& \forall \ell\in\set{1,\dots,d_2}, \forall m\in\set{1,\dots,d_1}.
	\end{align}
Equation \eqref{eq:tilde-C-2-1} is the same as \eqref{eq:tilde-C-ell-m-elements} for $i=2$.
Using \eqref{eq:perturbation-summation-2-1} in \eqref{eq:x2-solution-1} gives
	\begin{align*}
	x_2(t) &= L_{2}^{t} x_2 + L_{2}^{t-1} V_{2} \left(\tilde C_{2,1} - \Lambda_{2}^{-t} \tilde C_{2,1} \Lambda_{1}^{t}\right)V_{1}^{-1} \msf{pert}_1(x_1) \\
	&= L_{2}^{t}\Big( x_2 + L_{2}^{-1} V_{2} \tilde C_{2,1}V_{1}^{-1} \msf{pert}_1(x_1)\Big) - L_{2}^{t-1} V_{2}\Lambda_{2}^{-t} \tilde C_{2,1} \Lambda_{1}^{t}V_{1}^{-1} \msf{pert}_1(x_1)
	\end{align*}
Since $L_{2}^{t-1} V_{2}\Lambda_{2}^{-t} = L_{2}^{-1}V_{2}$ and $\Lambda_{1}^{t}V_{1}^{-1} = V_{1}^{-1} L_{1}^{t}$, we get
	\begin{equation}
	x_2(t) =  L_{2}^{t}\Big( x_2 + L_{2}^{-1} V_{2} \tilde C_{2,1}V_{1}^{-1} \msf{pert}_1(x_1)\Big) - L_{2}^{-1}V_{2} \tilde C_{2,1} V_{1}^{-1} L_{1}^{t} \msf{pert}_1(x_1).
	\end{equation}
Defining $D_{2,1}$ as
	\begin{equation}\label{eq:D-2-1}
	D_{2,1} = L_{2}^{-1} V_{2} \tilde C_{2,1}V_{1}^{-1}
	\end{equation}
and $\msf{pert}_{2} : \C^{d_{1}}\times \C^{d_{2}} \to \C^{d_{2}}$ as
	\begin{equation}
	\msf{pert}_{2}(x_1,x_2) = x_2 + D_{2,1} \msf{pert}_1(x_1)
	\end{equation}
gives
	\begin{align*}
	x_{2}(t) &= L_{2}^{t} \msf{pert}_{2}(x_1,x_2) - D_{2,1} L_1^{t} \msf{pert}_{1}(x_1) \\
	&= (-1)^{0} D_{2,2} L_{2}^{t} \msf{pert}_{2}(x_1,x_2) - D_{2,1} L_1^{t} \msf{pert}_{1}(x_1)
	\end{align*}
sine $D_{2,2} = I_{d_2}$ be definition. Finally,
	\begin{equation}
	x_{2}(t) = \sum_{s=0}^{1} (-1)^{s} D_{2,2-s} L_{2-s}^{t}\msf{pert}_{2-s}(x_1,\dots, x_{2-s})
	\end{equation}
Using the change of variables $j = 2 - s$, we have that
	\begin{equation}\label{eq:x2-summation-solution}
	x_{2}(t) = \sum_{j=1}^{2} (-1)^{2-j} D_{2,j} L_{j}^{t}\msf{pert}_{j}(x_1,\dots, x_{j})
	\end{equation}
Equations \eqref{eq:D-2-1} - \eqref{eq:x2-summation-solution} are equivalent to equations \eqref{eq:i-th-system-solution-linear-cascade}, \eqref{eq:D-i-iminusj}, and \eqref{eq:pert-iminusj}, for $j=2$. 

\item[(\textbf{Induction step}):]
Assume \eqref{eq:i-th-system-solution-linear-cascade} --  \eqref{eq:pert-iminusj} hold for for all $j \leq i$ where $i\in\set{2,\dots, n-1}$. We show they hold for $i+1$ as well. 

Write $x_{i+1}(t) = \Pi_{i+1}\msf{Lin}^{\circ t}(x_1,\dots,x_n)$. By lemma \ref{lemma:general-solution}, eq. \eqref{eq:linear-chained-cascade-general-solution}, the solution is
	\begin{align*}
	x_{i+1}(t) &= L_{i+1}^{t} x_{i+1} + L_{i+1}^{t-1} V_{i+1} \sum_{k=0}^{t-1} \Lambda_{i+1}^{-k} V_{i+1}^{-1} C_{i+1,i} x_{i}(k) .
	\end{align*}
By the induction hypothesis
	\begin{equation}
	x_{i}(k) = \sum_{j=1}^{i} (-1)^{i-j} D_{i,j} L_{j}^{k} \msf{pert}_{j}(x_1,\dots,x_{j}).
	\end{equation}
which gives that $x_{i+1}(t)$ is (after interchanging the finite sums)
	\begin{align*}
	x_{i+1}(t) &= L_{i+1}^{t} x_{i+1} + \sum_{j=1}^{i} (-1)^{i-j} L_{i+1}^{t-1} V_{i+1} \sum_{k=0}^{t-1} \Lambda_{i+1}^{-k} V_{i+1}^{-1} C_{i+1,i}  D_{i,j} L_{j}^{k} \msf{pert}_{j}(x_1,\dots,x_{j}).
	\end{align*}
Since $L_{j}$ is diagonalizable, we substitute $V_{j} \Lambda_{j}^{k} V_{j}^{-1}$ for $L_{j}^{k}$ in the above equation to get
	\begin{equation}\label{eq:xi-induction-step-sum}
	\begin{aligned}
	x_{i+1}(t) &= L_{i+1}^{t} x_{i+1} \\
	&\quad+ \sum_{j=1}^{i} (-1)^{i-j} L_{i+1}^{t-1} V_{i+1} 
	\left(\sum_{k=0}^{t-1} \Lambda_{i+1}^{-k} V_{i+1}^{-1} C_{i+1,i}  D_{i,j} V_{j} \Lambda_{j}^{k} \right)V_{j}^{-1}\msf{pert}_{j}(x_1,\dots,x_{j}) .
	\end{aligned} 
	\end{equation}
By lemma \ref{lemma:perturbation-terms}, eq.\ \eqref{eq:perturbation-evolution}, with $B \equiv V_{i+1}^{-1} C_{i+1,i}  D_{i,j} V_{j}$ we have
	\begin{equation}\label{eq:sum-lambda-k-i+1-lambdak}
	\sum_{k=0}^{t-1} \Lambda_{i+1}^{-k} V_{i+1}^{-1} C_{i+1,i}  D_{i,j} V_{j} \Lambda_{j}^{k}
	= \tilde C_{i+1,j} - \Lambda_{i+1}^{-t}\tilde C_{i+1,j} \Lambda_{j}^{t},
	\end{equation}
where for $j \in \set{1,\dots, i}$ the matrix $\tilde C_{i+1,j} \in \C^{d_{i+1}\times d_{j}} $ has elements
	\begin{align}\label{eq:tilde-C-induction-step}
	\left[\tilde C_{i+1,j}\right]_{\ell,m} = \left[V_{i+1}^{-1} C_{i+1,i}  D_{i,j} V_{j}\right]_{\ell,m} \left(1 - \frac{\lambda_{j,m}}{\lambda_{i+1,\ell}}\right)^{-1}.
	\end{align}
Equation \eqref{eq:tilde-C-induction-step} is \eqref{eq:tilde-C-ell-m-elements} for $i+1$. Plugging \eqref{eq:sum-lambda-k-i+1-lambdak} into \eqref{eq:xi-induction-step-sum} gives
	\begin{align*}
	x_{i+1}(t) 
	&= L_{i+1}^{t} x_{i+1} \\
	&\quad+ \sum_{j=1}^{i} (-1)^{i-j} L_{i+1}^{t-1} V_{i+1} 
	\left(\tilde C_{i+1,j} - \Lambda_{i+1}^{-t}\tilde C_{i+1,j} \Lambda_{j}^{t} \right)V_{j}^{-1}\msf{pert}_{j}(x_1,\dots,x_{j}) \\
	&= L_{i+1}^{t} x_{i+1} + \sum_{j=1}^{i} (-1)^{i-j} L_{i+1}^{t-1} V_{i+1} 
	\tilde C_{i+1,j} V_{j}^{-1}\msf{pert}_{j}(x_1,\dots,x_{j})\\
	&\quad+ \sum_{j=1}^{i} (-1)^{i+1-j} L_{i+1}^{t-1} V_{i+1} 
	 \Lambda_{i+1}^{-t}\tilde C_{i+1,j} \Lambda_{j}^{t} V_{j}^{-1}\msf{pert}_{j}(x_1,\dots,x_{j})  \\
	&= L_{i+1}^{t} \left[ x_{i+1} + \sum_{j=1}^{i} (-1)^{i-j} L_{i+1}^{-1} V_{i+1} 
	\tilde C_{i+1,j} V_{j}^{-1}\msf{pert}_{j}(x_1,\dots,x_{j})\right]\\
	&\quad+ \sum_{j=1}^{i} (-1)^{i+1-j} L_{i+1}^{t-1} V_{i+1} 
	 \Lambda_{i+1}^{-t}\tilde C_{i+1,j} \Lambda_{j}^{t} V_{j}^{-1}\msf{pert}_{j}(x_1,\dots,x_{j}).
	\end{align*} 
Since $L_{i+1}^{t-1} V_{i+1} \Lambda_{i+1}^{-t} = L_{i+1}^{-1} V_{i+1}$ and $\Lambda_{j}^{t} V_{j}^{-1} = V_{j}^{-1} L_{j}^{t}$, then
	\begin{align*}
	x_{i+1}(t) 
	&= L_{i+1}^{t} \left[ x_{i+1} + \sum_{j=1}^{i} (-1)^{i-j} \left(L_{i+1}^{-1} V_{i+1} 
	\tilde C_{i+1,j} V_{j}^{-1}\right) \msf{pert}_{j}(x_1,\dots,x_{j})\right]\\
	&\quad+ \sum_{j=1}^{i} (-1)^{i+1-j} \left(L_{i+1}^{-1} V_{i+1} \tilde C_{i+1,j} V_{j}^{-1}\right) L_{j}^{t}\msf{pert}_{j}(x_1,\dots,x_{j}).
	\end{align*} 
For $j=1,\dots, i$, define
	\begin{equation}
	D_{i+1,j} = L_{i+1}^{-1} V_{i+1}  \tilde C_{i+1,j} V_{j}^{-1} 
	\end{equation}
as in eq.\ \eqref{eq:D-i-iminusj} and $\msf{pert}_{i+1} : \C^{d_1}\times \cdots \times \C^{d_{i+1}} \to \C^{d_{i+1}}$ as
	\begin{align*}
	\msf{pert}_{i+1}(x_1,\dots, x_n) 
	&= x_{i+1} + \sum_{j=1}^{i} (-1)^{i-j} L_{i+1}^{-1} V_{i+1} 
		\tilde C_{i+1,j} V_{j}^{-1}\msf{pert}_{j}(x_1,\dots,x_{j}) \\
	&= x_{i+1} + \sum_{j=1}^{i} (-1)^{i-j} D_{i+1,j}\msf{pert}_{j}(x_1,\dots,x_{j}) 
	\end{align*}
as in eq.\ \eqref{eq:pert-iminusj} with the substitution $i \mapsto i+1$. Substituting these definitions into the expression for the solution $x_{i+1}(t)$ and defining $D_{i+1,i+1} = I_{d_{i+1}}$, we have
	\begin{align*}
	x_{i+1}(t) 
	&= L_{i+1}^{t}\msf{pert}_{i+1}(x_1,\dots,x_{i+1})
	+ \sum_{j=1}^{i} (-1)^{i+1-j} D_{i+1,j} L_{j}^{t} \msf{pert}_{j}(x_1,\dots,x_{j}) \\
	&= (-1)^0 D_{i+1,i+1} L_{i+1}^{t}\msf{pert}_{i+1}(x_1,\dots,x_{i+1})
	+ \sum_{j=1}^{i} (-1)^{i+1-j} D_{i+1,j} L_{j}^{t} \msf{pert}_{j}(x_1,\dots,x_{j}) \\
	&= \sum_{j=1}^{i+1} (-1)^{i+1-j} D_{i+1,j} L_{j}^{t} \msf{pert}_{j}(x_1,\dots,x_{j}) .
	\end{align*}
Comparing with \eqref{eq:i-th-system-solution-linear-cascade} with the substitution $i\mapsto i+1$, we see that the induction is complete.
\end{enumerate}
This completes the proof.
\end{proof}

\begin{corollary}\label{lem:chained-linear-cascade-asymptotic-proportionality}
Assume that Condition \ref{eq:chained-cascade-conditions} holds for \eqref{eq:linear-chained-cascade}. Then for all $i\in\set{2,\dots,n}$ and $t \in \N$,
	\begin{align}\label{eq:level-i-chained-linear-solution-bound}
	\norm*{\Pi_i \circ \msf{Lin}^{\circ t}(x_1, \dots, x_n) - L_i^{t}( \msf{pert}_{i}(x_1, \dots, x_i) ) }_{C^{d_i}} \leq \sum_{j=1}^{i-1} \norm{D_{i,j}}\norm{L_{j}^t \msf{pert}_{j}(x_1,\dots,x_{j})}_{\C^{d_{j}}}
	\end{align}
where $D_{i,j}$ and $\msf{pert}_{j}$ are given by \eqref{eq:D-i-iminusj} and \eqref{eq:pert-iminusj}. Furthermore, 
	\begin{align}\label{eq:lemma-relative-error-epsilon-limit-bound}
	\lim_{t\to\infty} \frac{\norm[\big]{\Pi_i \circ \msf{Lin}^{\circ t}(x_1,\dots, x_n) - L_i^{t} (\msf{pert}_{i}(x_1, \dots, x_i) ) }}{\norm{L_i}^t} = 0.
	\end{align}
\end{corollary}

\begin{proof}
Inequality \eqref{eq:level-i-chained-linear-solution-bound} follows directly from lemma \ref{lemma:level-i-solution-with-perturbation-terms}, eq.\ \eqref{eq:i-th-system-solution-linear-cascade} and the fact that $D_{i,i} = I_{d_{i}}$. Equation \eqref{eq:lemma-relative-error-epsilon-limit-bound} follows from the conditions $\norm{L_{j}} < \norm{L_{i}}$ for $j=1,\dots, i-1$.
\end{proof}

\subsection{Perturbation of principal eigenfunctions: nominal, linear system}\label{sec:principal-eigenfunction-perturbation-proof}
We now prove theorem \ref{thrm:perturbation-principal-eigenfunctions}. It is a straightforward application of theorem \ref{thrm:linear-asymptotic-equivalence}.

\begin{proof}[Proof of theorem \ref{thrm:perturbation-principal-eigenfunctions}]
We first show that for $i \geq 1$ and $t \geq 0$, that \eqref{eq:thrm-linear-principal-eigenfunction-proportionality} holds.

By definition
	\begin{align*}
	\mc U_{\msf{Lin}}^{\circ t} \psi_{(0,\dots,0,s_i,0,\dots,0)}(x_1,\dots, x_n) 
	&= \psi_{(0,\dots,0,s_i,0,\dots,0)}(\msf{Lin}^{\circ t}(x_1,\dots, x_n)) \nonumber \\
	&= \psi_{i,s_i}(\Pi_i \circ \msf{Lin}^{\circ t}(x_1,\dots, x_n)) \\
	&= \psi_{i,s_i}(\Pi_i\circ\msf{Nom}^{\circ t}(\msf{pert}(x_1,\dots, x_n))) \\
	&\quad+ \psi_{i,s_i}(\Pi_i \circ \msf{Lin}^{\circ t}(x_1,\dots, x_n) - \Pi_i \circ \msf{Nom}^{\circ t}(\msf{pert}(x_1,\dots, x_n))) \\
	&= (\mc U_{\msf{Nom}}^{\circ t}(\psi_{i,s_i}\circ \Pi_i))\circ \msf{pert}(x_1,\dots, x_n) \\
	&\quad+ \psi_{i,s_i}(\Pi_i \circ \msf{Lin}^{\circ t}(x_1,\dots, x_n) - \Pi_i \circ \msf{Nom}^{\circ t}(\msf{pert}(x_1,\dots, x_n))) \\
	&= (\mc U_{\msf{Nom}}^{\circ t}\psi_{(0,\dots,0,s_i,0,\dots,0)})\circ\msf{pert}(x_1,\dots, x_n) \\
	&\quad+ \psi_{i,s_i}(\Pi_i \circ \msf{Lin}^{\circ t}(x_1,\dots, x_n) - \Pi_i \circ \msf{Nom}^{\circ t}(\msf{pert}(x_1,\dots, x_n))) 
	\end{align*}
Therefore,
	\begin{align}
	&\abs*{ \mc U_{\msf{Lin}}^{\circ t} \psi_{(0,\dots,0,s_i,0,\dots,0)}(x_1,\dots, x_n) - \mc U_{\msf{Nom}}^{\circ t}\psi_{(0,\dots,0,s_i,0,\dots,0)}(\msf{pert}(x_1,\dots, x_n))} \nonumber \\
	&\qquad\leq \norm{\psi_{i,s_i}} \norm*{\Pi_i \circ \msf{Lin}^{\circ t}(x_1,\dots, x_n) - \Pi_i\circ \msf{Nom}^{\circ t}(\msf{pert}(x_1,\dots, x_n)) } \label{eq:linear-principal-eigenfunction-bound}
	\end{align}
By theorem \ref{thrm:linear-asymptotic-equivalence}, eq.\ \eqref{eq:thrm-asymptotic-proportionality}, for all $t \geq 0$ and $i\geq 1$,
	\begin{align}
	\norm*{\Pi_i \circ \msf{Lin}^{\circ t}(x_1,\dots, x_n) - \Pi_i \circ \msf{Nom}^{\circ t}(\msf{pert}(x_1,\dots, x_n)) } \leq 
	\sum_{j=1}^{i-1} \norm{D_{i,j}}\norm{L_{j}^t \msf{pert}_{j}(x_1,\dots,x_{j})}
	\end{align}
This estimate along with \eqref{eq:linear-principal-eigenfunction-bound} gives \eqref{eq:thrm-linear-principal-eigenfunction-proportionality}. 

By theorem \ref{thrm:linear-asymptotic-equivalence}, eq.\ \eqref{eq:thrm-asymptotic-relative-error}, for all $i\in\set{1,\dots, n}$ and any $\eps > 0$,
	\begin{align}
	\frac{\norm*{\Pi_i \circ \msf{Lin}^{\circ t}(x_1,\dots, x_n) - \Pi_i \circ \msf{Nom}^{\circ t}(\msf{pert}(x_1,\dots, x_n)) }}{\norm{L_i}^{t}} \leq \eps 
	\end{align}
for all $t$ large enough. This is equivalent to \eqref{eq:thrm-linear-principal-eigenfunction-0-asymptotic-relative-error}.
\end{proof}

\subsection{Asymptotic equivalence for nonlinear, chained cascade}\label{sec:nonlinear-chained-cascade-asymptotic-equivalence-proof}

We now prove Theorem \ref{thrm:nonlinear-chained-cascade-asymptotic-equivalence}. It is a straight-forward application of Theorem  \ref{thrm:linear-asymptotic-equivalence} and that the topological conjugacy is a homeomorphism.

\begin{proof}[Proof of Theorem \ref{thrm:nonlinear-chained-cascade-asymptotic-equivalence}]
Fix $\eps > 0$ and $\vec y = (y_1,\dots, y_n) \in \C^{d_1}\times\cdots\times \C^{d_n}$. Define $\vec x = \tau^{-1}(\vec y)$. Denote by $\overline{B_{i}}$ the closed unit ball of radius centered at the origin in $\C^{d_i}$. Condition \ref{eq:chained-cascade-conditions} and Theorem \ref{thrm:linear-asymptotic-equivalence}, eq.\ \eqref{eq:thrm-asymptotic-relative-error} guarantee that $\msf{Lin}^{\circ t}(\vec x)$ and $\msf{Nom}^{\circ t}(\msf{pert}(\vec x))$ are in the compact set $\overline{B_{1}} \times \cdots \times \overline{B_{n}}$ for all $t$ large enough.

Since $\tau$ is a continuous, it is uniformly continuous on $\overline{B_{1}} \times \cdots \times \overline{B_{r}}$. Let $\delta > 0$ be such that if $\vec{x}, \vec{x}' \in \overline{B_1} \times \cdots \times \overline{B_n}$ and $\norm{\vec{x} - \vec{x}'} < \delta$, then $\norm{\tau(\vec{x}) - \tau(\vec{x}')} < \eps$.

By corollary \ref{cor:asymptotic-equivalence-cascade-system-absolute-error}, there is a $T \in \N$ such that $t \geq T$ implies
	\begin{align}
	\norm{\msf{Lin}^{\circ t}(\vec x) - \msf{Nom}^{\circ t}(\msf{pert}(\vec x)) }_{\times} < \delta.
	\end{align}
The uniform continuity of $\tau$ implies that 
	\begin{align}\label{eq:tau-forward-image-asymptotic-equivalence}
	\norm{\tau \circ \msf{Lin}^{\circ t}(\vec x) - \tau\circ\msf{Nom}^{\circ t}(\msf{pert}(\vec x)) }_{\times} &< \eps, &(\forall t \geq T).
	\end{align}
Now, since $\tau$ is a topological conjugacy, $\msf{Lin}^{\circ t} = \tau^{-1}\circ \msf{NonLin}^{\circ t} \circ \tau$. Plugging this in into \eqref{eq:tau-forward-image-asymptotic-equivalence} gives, for all $t \geq T$,
	\begin{align*}
	\eps &> \norm{\tau \circ (\tau^{-1}\circ \msf{NonLin}^{\circ t} \circ \tau)(\vec x) - \tau\circ\msf{Nom}^{\circ t}(\msf{pert}(\vec x)) }_{\times} \\
	&= \norm{(\msf{NonLin}^{\circ t}(\tau(\vec x)) - \tau\circ\msf{Nom}^{\circ t}(\msf{pert}(\vec x)) }_{\times} \\
	&= \norm{(\msf{NonLin}^{\circ t}(\tau(\vec x)) - (\tau\circ\msf{Nom}^{\circ t}\circ \tau^{-1})\circ \tau \circ(\msf{pert}(\vec x)) }_{\times} \\
	&= \norm{(\msf{NonLin}^{\circ t})(\tau(\tau^{-1}(\vec y))) - (\tau\circ\msf{Nom}^{\circ t}\circ \tau^{-1})\circ \tau \circ(\msf{pert}(\tau^{-1}(\vec y))) }_{\times} \\
	&= \norm{\msf{NonLin}^{\circ t}(\vec y) - (\tau\circ\msf{Nom}^{\circ t}\circ \tau^{-1})\circ (\tau \circ \msf{pert} \circ \tau^{-1})(\vec y) }_{\times} .
	\end{align*}
Therefore,
	\begin{align*}
	\lim_{t\to\infty} \norm{\msf{NonLin}^{\circ t}(\vec y) - (\tau\circ\msf{Nom}^{\circ t}\circ \tau^{-1})\circ (\tau \circ \msf{pert} \circ \tau^{-1})(\vec y) }_{\times} = 0.
	\end{align*}
This completes the proof.
\end{proof}

\subsection{Perturbation of principal eigenfunctions: nominal, nonlinear cascade}\label{sec:perturbation-nonlinear-principal-eigenfunctions-proof}

To save space in the following proof, we will write $\psi_{(0,\cdots,0,s_i,0,\dots,0)}$ as $\psi_{s_i\hat e_{n,i}}$, where $\hat e_{n,i}$ is the $i$-th canonical basis vector of length $n$. 

\begin{proof}[Proof of Theorem \ref{thrm:peturbation-nonlinear-principal-eigenfunctions}]
Fix $\vec y = (y_1,\dots, y_n) \in \C^{d_1}\times\cdots\times\C^{d_n}$ and let $\vec x = \tau^{-1}(\vec y)$. The topological conjugacy satisfies
	\begin{equation}
	\msf{Lin}^{\circ t}(\vec x) = (\tau^{-1} \circ \msf{NonLin}^{\circ t} \circ \tau)(\vec x).
	\end{equation}
Using this relation, we get
	\begin{align*}
	\mc U_{\msf{Lin}}^{\circ t} \psi_{s_i\hat e_{n,i}}(\vec x)
	&= \psi_{s_i\hat e_{n,i}}(\msf{Lin}^{\circ t}\vec x) \\
	&= \psi_{s_i\hat e_{n,i}}((\tau^{-1} \circ \msf{NonLin}^{\circ t} \circ \tau)(\vec x)) \\
	&= (\psi_{s_i\hat e_{n,i}}\circ \tau^{-1})(\msf{NonLin}^{\circ t}(\tau(\vec x))) \\
	&= \mc U_{\msf{NonLin}}^{\circ t}(\psi_{s_i\hat e_{n,i}}\circ \tau^{-1})(\vec y).
	\end{align*}
On the other hand,
	\begin{align*}
	\mc U_{\msf{Nom}}^{\circ t} \psi_{s_i\hat e_{n,i}}(\msf{pert}(\vec x)) 
	&= \psi_{s_i\hat e_{n,i}}(\msf{Nom}^{\circ t}(\msf{pert}(\vec x))) \\
	&= \psi_{s_i\hat e_{n,i}}(\tau^{-1}\circ \tau \circ \msf{Nom}^{\circ t}\circ \tau^{-1} \circ \tau(\msf{pert}(\vec x))) \\
	&= (\psi_{s_i\hat e_{n,i}}\circ\tau^{-1})((\tau \circ \msf{Nom}^{\circ t}\circ \tau^{-1}) \circ \tau(\msf{pert}(\vec x))) \\
	&= (\psi_{s_i\hat e_{n,i}}\circ\tau^{-1})((\tau \circ \msf{Nom}^{\circ t}\circ \tau^{-1}) (\tau \circ \msf{pert} \circ \tau^{-1})(\vec y)) \\
	&= \mc U_{\tau \circ \msf{Nom}^{\circ t}\circ \tau^{-1}}^{\circ t}(\psi_{s_i\hat e_{n,i}}\circ\tau^{-1})((\tau \circ \msf{pert} \circ \tau^{-1})(\vec y)) .
	\end{align*}
Combining these two expression, we have
	\begin{align*}
	&\abs*{\mc U_{\msf{NonLin}}^{\circ t}(\psi_{s_i\hat e_{n,i}}\circ \tau^{-1})(\vec y) - \mc U_{\tau \circ \msf{Nom}^{\circ t}\circ \tau^{-1}}^{\circ t}(\psi_{s_i\hat e_{n,i}}\circ\tau^{-1})((\tau \circ \msf{pert} \circ \tau^{-1})(\vec y))} \\
	&\qquad\qquad= \abs*{\mc U_{\msf{Lin}}^{\circ t} \psi_{s_i\hat e_{n,i}}(\vec x)  -  \mc U_{\msf{Nom}}^{\circ t} \psi_{s_i\hat e_{n,i}}(\msf{pert}(\vec x)) }.
	\end{align*}
Theorem \ref{thrm:perturbation-principal-eigenfunctions}, eq.\ \eqref{eq:thrm-linear-principal-eigenfunction-0-asymptotic-relative-error}, implies
	\begin{align*}
	0 &= 
	\lim_{t\to\infty} \frac{\abs*{\mc U_{\msf{Lin}}^{\circ t} \psi_{s_i\hat e_{n,i}}(\vec x)  -  \mc U_{\msf{Nom}}^{\circ t} \psi_{s_i\hat e_{n,i}}(\msf{pert}(\vec x)) } }{\norm{L_i}^t} \\
	&=\lim_{t\to\infty} \frac{\abs*{\mc U_{\msf{NonLin}}^{\circ t}(\psi_{s_i\hat e_{n,i}}\circ \tau^{-1})(\vec y) - \mc U_{\tau \circ \msf{Nom}^{\circ t}\circ \tau^{-1}}^{\circ t}(\psi_{s_i\hat e_{n,i}}\circ\tau^{-1})((\tau \circ \msf{pert} \circ \tau^{-1})(\vec y))} }{\norm{L_i}^t} .
	\end{align*}
\end{proof}

\section{Conclusions}\label{sec:conclusions}
In this paper, we have analyzed the Koopman spectrum of cascaded dynamical systems; systems formed by wiring component subsystems together in a lower block triangular form. We show the existence of a perturbation function that maps initial conditions to initial conditions such that the evolution of the cascaded system from an initial condition is asymptotically equivalent to the evolution due to each component subsystem (decoupled from the others) if the initial conditions for the component subsystem are given by the perturbation function applied to the cascaded system's initial condition. The rate of convergence of the orbits in each subsystem is faster than the the decay rate of the decoupled subsystem to its fixed point. This is captured in our results by saying that each subsystem has zero asymptotic relative error. We also show a bound on the distance between the trajectories that hold for all time. From these results, it follows that the Koopman principal eigenvalues of each subsystem are also Koopman eigenvalues of the cascaded system. The principal eigenfunctions for a subsystem become eigenfunctions for the cascaded system when composed with the perturbation function. It also follows these results hold for a nonlinear cascade if it is topologically conjugate to a linear cascade for which the results hold.

These results are useful in the analysis of large interconnected systems. Often, in order to analyze these systems efficiently, a decomposition into a lower block diagonal form (cascade structure) or block diagonal form must be performed. Various techniques have been proposed to do this, with more research begin currently done. Once in this structure, the results of this paper allow a further analysis of the system without having to simulate it. The results on the principal Koopman eigenfunctions tell, \textit{a priori}, how any observable of interest on the system would behave under the dynamics. Such observables would only need to be expanded into the principal eigenfunctions and their products, whose eigenvalues are given by products of the principal eigenvalues.

\section*{Acknowledgments}

This research was partially funded under a subcontract from HRL Laboratories, LLC under DARPA contract N66001-16-C-4053 and additionally funded by the DARPA Contract HR0011-16-C-0116.

The views expressed are those of the authors and do not reflect the official policy or position of the Department of Defense or the U.S. Government. Distribution Statement `A': Approved for Public Release, Distribution Unlimited.



\end{document}